\documentclass[11pt]{amsart}
\usepackage{amssymb}
\usepackage{amsmath}
\usepackage{amsthm}
\usepackage{amsrefs}
\usepackage{mathrsfs}
\usepackage{faktor}
\usepackage{tikz}

\usetikzlibrary{arrows}
\usetikzlibrary{shapes.geometric}
\usetikzlibrary{decorations.pathreplacing,decorations.markings}
\usepackage[utf8]{inputenc}
\usepackage[top=1in,bottom=1in,right=1in,left=1in]{geometry}
\usepackage{hyperref}
\usepackage{enumitem}
\usepackage{wasysym}

\usepackage[capitalise]{cleveref}

\newtheorem{theorem}{Theorem}[section]
\newtheorem{lem}[theorem]{Lemma}
\newtheorem{proposition}[theorem]{Proposition}
\newtheorem{cor}[theorem]{Corollary}

\theoremstyle{definition}
\newtheorem{dfn}[theorem]{Definition}
\newtheorem{ex}[theorem]{Example}
\newtheorem{rmk}[theorem]{Remark}
\newtheorem{ntn}[theorem]{Notation}

\numberwithin{theorem}{section}

\newenvironment{theorem_no_number}[1][]{\begin{trivlist}
\item[\hskip \labelsep {\bfseries Theorem \def\temp{#1}\ifx\temp\empty  #1\else  #1\fi
.}] \itshape}  {\end{trivlist}}

\DeclareMathOperator{\End}{End}

\newcommand{\ns}{\mathsf{NSPACE}}

\title{Quadratic Diophantine equations, the Heisenberg
group and formal languages}

\author{Alex Levine}

\address{Department of Mathematics, Alan Turing Building, University of Manchester, M13 9PL}

\email{alex.levine@manchester.ac.uk}

\keywords{quadratic Diophantine equations, equations in groups, EDT0L languages,
nilpotent groups}

\subjclass[2020]{03D05, 20F10, 20F65, 20F18, 68Q45, 11D09}

\begin{document}

\begin{abstract}
 	We express the solutions to quadratic equations with two variables in the
 	ring of integers using EDT0L languages. We use this to show that EDT0L
 	languages can be used to describe the solutions to one-variable equations in
 	the Heisenberg group. This is done by reducing the question of solving a
 	one-variable equation in the Heisenberg group to solving an equation in the
 	ring of integers, exploiting the strong link between the ring of integers and
	nilpotent groups.
\end{abstract}

\maketitle

\section{Introduction}
	Equations in nilpotent groups, and equations in the ring of integers are
	deeply linked. This has been demonstrated repeatedly since Roman'kov used
	Matijasevič's result that the satisfiability of systems of quadratic equations
	in integers is undecidable \cite{Matijasevic} to show that the satisfiability
	of systems of equations in various free nilpotent groups is undecidable
	\cite{romankov_undecidable_first}. The proofs that many other nilpotent groups
	have an undecidable satisfiability of equations involve a similar method of
	reducing the question to systems of quadratic equations in integers
	\cites{romankov_undecidable_class_2, random_nilpotent_eqns,
	duchin_liang_shapiro}. After this Matijasevič's result can be applied. Duchin,
	Liang and Shapiro's positive result that the satisfiability of single
	equations is decidable in class \(2\) nilpotent groups with a virtually cyclic
	commutator subgroup \cite{duchin_liang_shapiro}, involves reducing the problem
	to single quadratic equations in integers, and then applying Siegel's result
	that the satisfiability of such equations is decidable \cite{Siegel}. There
	are other aspects of equations that can be studied beyond decidability, such
	as the structure of the set of solutions, and one way of investigating this is
	to express the set of solutions as a formal language, which is the purpose of
	this paper.

	Formal languages have been used in group theory in a variety of settings for
	over the last few decades. Perhaps one of the most striking uses of languages
	in groups was Anisimov's result of 1971 showing that the set of words over a
	given finite generating set for a group \(G\) that represent the identity
	(called the \textit{word problem} of \(G\)) forms a regular language if and
	only if \(G\) is finite \cite{Anisimov}. Muller, Schupp and Dunwoody later
	showed that a group has a context-free word problem if and only if it is
	virtually free \cites{muller_schupp, fp_accessible}. Following Muller, Schupp
	and Dunwoody's result, the word problem has been generalised in a number of
	different ways, including to semigroups \cites{gilbert_noonanheale, kambites,
	gen_word_problem}.

	Languages have also been used in a number of other different settings. Groups
	that admit regular geodesic normal forms must have rational growth series.
	Regular languages can be used to describe \((\lambda, \ \mu)\)-quasi-geodesics
	in hyperbolic groups if \(\lambda\) and \(\mu\) are rational
	\cite{hyp_quasigeo_reg}. Languages have also been used to study conjugacy in
	various classes of groups \cites{conj_langs, CF_conj}. The complement of the
	word problem has also been studied for a wide variety of groups, including
	Thompson's groups, the Grigorchuk group and Baumslag-Solitar groups
	\cites{bishop_elder_journal, appl_L_systems_GT, CF_coWP, coCF_univ, coCF_V}.

	In 2016, Ciobanu, Diekert and Elder showed that the solutions to a system of
	equations in a free group can be expressed as an EDT0L language
	\cite{eqns_free_grps}. This led to a number of results showing that solutions
	to systems of equations in various classes of groups are EDT0L, starting with
	right-angled Artin groups in the same year \cite{EDT0L_RAAGs}. Virtually free
	groups \cite{VF_eqns}, hyperbolic groups \cite{eqns_hyp_grps}, virtually
	abelian groups \cite{VAEP}, and virtually direct products of hyperbolic groups
	\cite{EDT0L_extensions} all followed later.

  We consider single equations in the Heisenberg group in one variable. The fact
  that satisfiability of equations with one variable in the Heisenberg group is
  decidable was first shown by Repin \cite{repin}. Duchin, Liang and Shapiro
  generalised this to all single equations in any number of variables in a class
  \(2\) nilpotent groups with a virtually cyclic commutator subgroup
  \cite{duchin_liang_shapiro}, however Roman'kov showed that the restriction on
  the commutator subgroup cannot be relaxed \cite{romankov_undecidable_class_2}.
  We show that the solutions to these equations, when written as words in
  Mal'cev normal form are EDT0L, with an EDT0L system constructible in
  non-deterministic polynomial space.

	\begin{theorem_no_number}[\ref{Heisenberg_eqns_EDT0L_thm}]
		Let \(L\) be the solution language to a single equation with one
		variable in the Heisenberg group, with respect to the Mal'cev generating
		set and normal form. Then
		\begin{enumerate}
			\item The language \(L\) is EDT0L;
			\item An EDT0L system for \(L\) is constructible in \(\ns(n \mapsto
			n^8(\log n)^2)\), where the input size is the length of the equation as
			an element of \(H(\mathbb{Z}) \ast F(X)\).
		\end{enumerate}
	\end{theorem_no_number}

	Proving Theorem \ref{Heisenberg_eqns_EDT0L_thm} involves reducing the problem
	of solving one-variable equations in the Heisenberg group to describing
	solutions to two-variable quadratic equations in the ring of integers. This
	uses a similar construction to the method of Duchin, Liang and Shapiro, which
	was used to show that the satisfiability of single equations in any class
	\(2\) nilpotent group with a virtually cyclic commutator subgroup is decidable
	\cite{duchin_liang_shapiro}.

	Despite the extensive use EDT0L languages have had in describing solutions to
	group equations, there have been no attempts to describe solutions to
	equations in the ring of integers using EDT0L languages, other than linear
	equations, which are just equations in an abelian group. In order to make
	progress studying equations in the Heisenberg group, we will have to first
	learn to what extent EDT0L languages can be used to describe solutions to
	quadratic equations in the ring of integers. Our result for equations in the
	Heisenberg group involves reducing to the two-variable case of quadratic
	equations in integers.

	\begin{theorem_no_number}[\ref{quad_eqn_2_var_integers_EDT0L_thm}]
		Let
		\begin{equation}
			\label{intro_quad_eqn}
			\alpha X^2 + \beta XY + \gamma Y^2 + \delta X + \epsilon Y + \zeta = 0
		\end{equation}
		be a two-variable quadratic
		equation in the ring of integers, with a set \(S\) of solutions. Then
		\begin{enumerate}
			\item The language \(L = \{a^x \# b^y \mid (x, \ y) \in S\}\) is EDT0L
			over the alphabet \(\{a, \ b, \ \#\}\);
			\item Taking the input size to be \(\max(|\alpha|, \ |\beta|, \ |\gamma|,
			\ |\delta|, \ |\epsilon|, \ |\zeta|)\), an EDT0L system for \(L\) is
			constructible in \(\ns(n \mapsto n^4 \log n)\).
		\end{enumerate}
	\end{theorem_no_number}

	We prove this theorem using Lagrange's method. This involves reducing an
	arbitrary two-variable quadratic equation to a generalised Pell's equation
	\(X^2 - DY^2 = N\). This again reduces to Pell's equation \(X^2 - DY^2 = 1\),
	the set of solutions of which is well-understood. The reduction involves
	writing solutions to the two variable quadratic equation
	\eqref{intro_quad_eqn} in the form \(\frac{\lambda x + \mu y + \xi}{\eta}\),
	where \((x, \ y)\) is a solution to some computable Pell's equation, and
	\(\lambda, \ \mu, \ \xi, \ \eta \in \mathbb{Z}\) with \(\eta \neq 0\) are all
	computable.

	Showing that the set of solutions to Pell's equation can be expressed as an
	EDT0L language is not too difficult. However, studying \(\frac{\lambda x + \mu
	y + \xi}{\eta}\) requires more work, particularly when the signs of
	\(\lambda\), \(\mu\) and \(\xi\) are not all the same, or when \(|\eta| \geq
	2\). To deal with the division, we use the concept of \(\#\)-separated EDT0L
	systems, first introduced in \cite{EDT0L_extensions}, and work in the world of
	EDT0L languages.

	Understanding \(\lambda x + \mu y + \xi\), when \(\lambda\), \(\mu\) and
	\(\xi\) are not all the same sign is more difficult to resolve by manipulating
	EDT0L systems. This is because we represent the integer \(n\) by \(a^n\); that
	is, a word of length \(n\) comprising \(n\) occurrences of the letter \(a\)
	(when \(n \geq 0\)) or \(n\) occurrences of the letter \(a^{-1}\) (when \(n
	\leq 0\)). Adding \(4\) to \(-2\) corresponds to concatenating \(a^4\) with
	\(a^{-2}\), resulting in \(a^4 a^{-2}\), which is not equal as a word to
	\(a^2\). We cannot simply `cancel' \(a\)'s and \(a^{-1}\)s either; in general
	the language obtained by freely reducing all words in an EDT0L language is not
	EDT0L (it need not even be recursive). Therefore, we work with facts about the
	solutions themselves to show that for fixed integers \(\lambda\), \(\mu\) and
	\(\xi\), the set
	\[
		\{\lambda x + \mu y + \xi \mid (x, \ y) \text{ is a solution to } X^2
		- DY^2 = 1\}
	\]
	is sufficiently well-behaved that we can describe it using an EDT0L language.
	We can then apply our method for the `division' to obtain the desired
	language.

	EDT0L languages were defined by Rozenberg in 1973
	\cite{ET0L_EDT0L_def_article} as members of the broad collection of languages
	called L-systems. L-systems were introduced by Lindenmayer for the study of
	growth of organisms. A key aspect of L-systems is their ability to perform
	parallel computation, which was useful for the study of organisms, but has
	also been effective for expressing solutions to equations, as the solutions to
	each individual variable can be computed in parallel. EDT0L systems interested
	computer scientists in the 1970s, with a variety of papers proving different
	results from pumping lemmas to alternative definitions \cites{CX0L_det,
	EPDT0L_to_EDT0L, CF_not_EDT0L}. Since Ciobanu, Diekert and Elder's paper
	showing solutions to systems of equations in free groups can be expressed as
	EDT0L languages, interest in the class has been reinvigorated, leading to a
	number of recent publications on EDT0L languages \cites{EDT0L_permuations,
	appl_L_systems_GT}, in addition to the previously mentioned papers on
	equations.

	We cover the preliminaries of the considered topics in Section
	\ref{prelim_sec}. In Section \ref{division_EDT0L_sec}, we prove our result
	about `division' of EDT0L languages by a constant that is a key part of the
	proof that the solutions to two-variable quadratic equations in the ring of
	integers are EDT0L, which appears in Section \ref{quad_eqns_sec}. In Section
	\ref{pell_eqn_sec}, we study the solutions to Pell's equation, and their
	images under linear functions. The proof of the fact that solutions to
	two-variable quadratic equations are EDT0L involves reducing to the case of
	Pell's equation. This reduction is contained in Section \ref{quad_eqns_sec}.
	Section \ref{Heisenberg_eqns_sec} includes the reduction from equations in the
	Heisenberg group to quadratic equations in the ring of the integers, and the
	proof that single equations in one variable in the Heisenberg group are
	expressible as EDT0L languages.

	\begin{ntn}
		We introduce a variety of notation we will frequently use.
		\begin{enumerate}
			\item Functions will be written to the right of their arguments;
			\item If \(S\) is a subset of a group, we define \(S^\pm = S \cup
			S^{-1}\);
			\item We use \(\varepsilon\) to denote the empty word;
			\item For elements \(g\) and \(h\) of a group \(G\), the commutator is
			defined by \([g, \ h] = g^{-1}h^{-1}gh\).
		\end{enumerate}
	\end{ntn}

\section{Preliminaries}
	\label{prelim_sec}

	\subsection{Nilpotent groups}

	We start with the definitions of a nilpotent group and the Heisenberg
	group. For a comprehensive introduction to nilpotent groups we refer the
	reader to \cite{theory_nil_groups}.

	\begin{dfn}
		Let \(G\) be a group. Define \(\gamma_i(G)\) for all \(i \in
		\mathbb{Z}_{\geq 0}\) inductively as follows:
		\begin{align*}
			& \gamma_0(G) = G \\
			& \gamma_i(G) = [G, \ \gamma_{i - 1}(G)] \text{ for } i > 0.
		\end{align*}
		The subnormal series \((\gamma_i(G))_i\) is called the \textit{lower central
		series} of \(G\). We call \(G\) \textit{nilpotent} of \textit{class} \(c\)
		if \(\gamma_c(G)\) is trivial.
	\end{dfn}

	\begin{dfn}
		The \textit{Heisenberg group} \(H(\mathbb{Z})\) is the class \(2\)
		nilpotent group defined by the presentation
		\[
			H(\mathbb{Z}) = \langle a, \ b, \ c \mid c = [a, \ b], \  [a, \ c] =
			[b, \ c] = 1 \rangle.
		\]
		Note that whilst the generator \(c\) is redundant, it is often easier to
		work with the generating set \(\{a, \ b, \ c\}\) than \(\{a, \ b\}\).

		The \textit{Mal'cev generating set} for the Heisenberg group is the set
		\(\{a, \ b, \ c\}\).
	\end{dfn}

	\subsection{Mal'cev normal form}

	  We now define the normal form that we will be using to represent our
	  solutions. This is used in \cite{duchin_liang_shapiro}, and we include the
	  proof of uniqueness and existence for completeness.

		The following facts about commutators in class \(2\) nilpotent groups will
		be used to induce the methods for `pushing' \(b\)s past \(a\)s in the
		Heisenberg group.

		\begin{lem}
			\label{commutator_lem}
			Let \(G\) be a class \(2\) nilpotent group, and \(g, \ h \in G\). Then
			\begin{enumerate}
				\item \([g^{-1}, \ h^{-1}] = [g, \ h]\),
				\item \([g^{-1}, \ h] = [g, \ h]^{-1}\).
			\end{enumerate}
		\end{lem}

		\begin{proof} For (1), since commutators are central,
				\[
					[g^{-1}, \ h^{-1}] = ghg^{-1} h^{-1} = ghg^{-1} h^{-1} gh h^{-1}
					g^{-1} = gh[g, \ h] h^{-1} g^{-1} = [g, \ h] ghh^{-1} g^{-1} = [g, \
					h].
				\]
				Similarly, for (2), we have
				\[
					[g^{-1}, \ h] = gh^{-1} g^{-1} h = g h^{-1} g^{-1} h g g^{-1} =
					g [g, \ h]^{-1} g^{-1} = gg^{-1} [g, \ h]^{-1} = [g, \ h]^{-1}.
				\]
		\end{proof}

		Using Lemma \ref{commutator_lem}, we now have a number of useful
		identities for `pushing' \(b\)s past \(a\)s in expressions over the
		Mal'cev generating set.

		\begin{lem}
			\label{malcev_gen_rules_lem}
			The following identities hold for the Mal'cev generators of the
			Heisenberg group:
			\begin{align*}
				& b a = a b c \\
				& b a^{-1} = a^{-1} b c^{-1} \\
				& b^{-1} a = a b^{-1} c^{-1} \\
				& b^{-1} a^{-1} = a^{-1} b^{-1} c.
			\end{align*}
		\end{lem}

		\begin{proof}
			We have
			\begin{align*}
				& b a = a b b^{-1} a^{-1} b a = a b c \\
				& b a^{-1} = a^{-1} b b^{-1} a b a^{-1} = a^{-1} b [b, \ a^{-1}]
				= a^{-1} b c^{-1} \\
				& b^{-1} a = a b^{-1} b a^{-1} b^{-1} a = a b^{-1} [b^{-1}, \ a]
				= a b^{-1} c^{-1} \\
				& b^{-1} a^{-1} = a^{-1} b^{-1} b a b^{-1} a^{-1} =
				a^{-1} b^{-1} [b^{-1}, \ a^{-1}] = a^{-1} b^{-1} c.
			\end{align*}
		\end{proof}

		The following lemma allows us to define the Mal'cev normal form for the
		Heisenberg group.

		\begin{lem}
			For each \(g \in H(\mathbb{Z})\) there exists a unique word of the form
			\(a^i b^j c^k\) that represents \(g\), where \(i, \ j, \ k \in
			\mathbb{Z}\).
		\end{lem}

		\begin{proof}
			Existence: Let \(w \in \{a, \ b, \ c, \ a^{-1}, \ b^{-1}, \
			c^{-1}\}^\ast\). To transform \(w\) into an equivalent word in the form
			\(a^i b^j c^k\), first note that \(c\) is central, so \(w\) is equal to
			\(u c^k\), where \(u \in \{a, \ b, \ a^{-1}, \ b^{-1}\}^\ast\), and \(k
			\in \mathbb{Z}\), which is obtained by pushing all \(c\)s and \(c^{-1}\)s
			in \(w\) to the right, then freely reducing. We can then look for any
			\(b\)s or \(b^{-1}\)s before \(a\)s or \(a^{-1}\)s, and use the rules of
			Lemma \ref{malcev_gen_rules_lem} to `swap' them, by adding a commutator.

			After doing these swaps, we can push the `new' \(c\)s and \(c^{-1}\)s to
			the back, to assume our word remains within \(\{a, \ b, \ a^{-1}, \
			b^{-1}\}^\ast (\{c\}^\ast \cup \{c^{-1}\}^\ast)\). By repeating this
			process, we will eventually have no more \(a\)s or \(a^{-1}\)s occurring
			after any \(b\) or \(b^{-1}\), and so will be in the form \(a^i b^j c^k\),
			where \(i, \ j, \ k \in \mathbb{Z}\).

			Uniqueness: Suppose \(i_1, \ i_2, \ j_1, \ j_2, \ k_1, \ k_2 \in
			\mathbb{Z}\) are such that \(a^{i_1} b^{j_1} c^{k_1} =_{H(\mathbb{Z})}
			a^{i_2} b^{j_2} c^{k_2}\). Then
			\begin{align*}
				1 & =_{H(\mathbb{Z})} a^{i_1} b^{j_1} c^{k_1} (a^{i_2} b^{j_2} c^{k_2})^{-1} \\
				& =_{H(\mathbb{Z})} a^{i_1} b^{j_1} c^{k_1} c^{-k_2} b^{-j_2} a^{-i_2} \\
				& =_{H(\mathbb{Z})} a^{i_1} b^{j_1 - j_2} a^{-i_2} c^{k_1 - k_2} \\
				& =_{H(\mathbb{Z})} a^{i_1} a^{-i_2} b^{j_1 - j_2} c^{-i_2(j_1 - j_2)}c^{k_1 - k_2} \\
				& =_{H(\mathbb{Z})} a^{i_1 - i_2} b^{j_1 - j_2} c^{-i_2(j_1 - j_2) + k_1 - k_2}.
			\end{align*}
			As \(1 \in \langle c \rangle\), we have that the above word lies in
			\(\langle c \rangle\). But since \(c\) commutes with \(a\) and \(b\),
			\(a^i b^j \in \langle c \rangle\) if and only if \(i = j = 0\). Thus \(i_1
			- i_2 = j_1 - j_2 = 0\). It follows that the above word equals \(c^{k_1 -
			k_2}\). Since this is a freely reduced word in \(\langle c \rangle\) as a
			power of \(c\), this represents the identity if and only if \(k_1 - k_2 =
			0\). Thus \(k_1 - k_2 = 0\), and the two words represent the same element
			of \(H(\mathbb{Z})\).
		\end{proof}

		\begin{dfn}
			The \textit{Mal'cev normal form} for the Heisenberg group is the normal
			form that maps an element \(g \in H(\mathbb{Z})\) to the unique word of
			the form \(a^i b^j c^k\), where \(i, \ j, \ k \in \mathbb{Z}\), that
			represents \(g\).
		\end{dfn}

\subsection{Space complexity}

	We give a short definition of space complexity. For a more detailed
	introduction, we refer the reader to \cite{computational_compl}.

	\begin{dfn}
		Let \(f \colon \mathbb{Z}_{\geq 0} \to \mathbb{Z}_{\geq 0}\). We say that an
		algorithm runs in \textit{non-deterministic \(f\)-space} (often written as
		\(\ns(f)\)) if it can be performed by a non-deterministic Turing machine
		with a read-only input tape, a write-only output tape, and a read-write work
		tape such that no computation path in the Turing machine uses	more than
		\(\mathcal{O}((n)f)\) units of the work tape, for an input of length \(n\).
	\end{dfn}

\subsection{Group equations}
	We start with the definition and some examples of equations in groups.

	\begin{dfn}
		Let \(G\) be a finitely generated group, \(V\) be a finite set and \(F(V)\)
		be the free group on \(V\). An \textit{equation} in \(G\) is an element \(w
		\in G \ast F(V)\) and denoted \(w = 1\). A \textit{solution} to \(w = 1\) is
		a homomorphism \(\phi \colon G \ast F(V) \to G\) that fixes elements of
		\(G\), such that \(w \phi = 1\). The elements of \(V\) are called the
		\textit{variables} of the equation. A \textit{system of equations} in \(G\)
		is a finite set of equations in \(G\), and a \textit{solution} to a system
		is a homomorphism that is a solution to every equation in the system.

		We say that two systems of equations in \(G\) are \textit{equivalent} if
		their sets of solutions are equal.

		Given a choice of generating set \(\Sigma\) for \(G\), we say the
		\textit{length} of \(w = 1\) is the length of the element \(w\) in
		\(G \ast F(V)\), with respect to the generating set \(\Sigma \cup V\).
		This will be our input size for any algorithm that takes a group equation
		as an input.
	\end{dfn}

	\begin{rmk}
		We will often abuse notation, and consider a solution to an equation in a
		group \(G\) to be a tuple of elements \((g_1, \ \ldots, \ g_n)\), rather
		than a homomorphism from \(G \ast F(X_1, \ \ldots, \ X_n) \to G\), where
		\(X_1, \ \ldots, \ X_n\) are variables. We can recover such a homomorphism
		\(\phi\) from a tuple by setting \(g \phi = g\) if \(g \in G\) and \(X_i
		\phi = g_i\). The action of \(\phi\) on the remaining elements is now
		determined as it is a homomorphism.
	\end{rmk}

	\begin{ex}
		Equations in the group \(\mathbb{Z}\) are linear equations in integers,
		and thus elementary linear algebra is sufficient to show that their
		satisfiability is decidable. A similar argument works for any finitely
		generated abelian group.

		For example, if we use \(a\) as the free generator for \(\mathbb{Z}\), then
		\(X^2 a^2 X^{-3} a Y^2 = 1\) is an equation in the group \(\mathbb{Z}\). We
		can rewrite this using additive notation to get
		\[
			2X + 2 - 3X + 1 + 2Y = 0.
		\]
		Using the fact that \(\mathbb{Z}\) is commutative, the above equation is
		equivalent to \(-X + 2Y + 3 = 0\). Thus the set of solutions can be written
		as
		\[
			\{(2y + 3, \ y) \mid y \in \mathbb{Z}\}.
		\]
	\end{ex}

\subsection{Equations in the ring of integers}
	We briefly define an equation in integers.

	\begin{dfn}
		An \textit{equation} in the ring of integers is an identity \((X_1, \
		\ldots, \ X_n)f = 0\), where \((X_1, \ \ldots, \ X_n)f \in \mathbb{Z}[X_1,
		\ \ldots, \ X_n]\) is a polynomial. The indeterminates \(X_1, \ \ldots, \
		X_n\) are called \textit{variables}. An equation is called
		\textit{quadratic} if the degree of \((X_1, \ \ldots, \ X_n)f\) is at most
		\(2\).

		A \textit{solution} to an equation \((X_1, \ \ldots, \ X_n)f = 0\) is a
		ring homomorphism \(\phi \colon \mathbb{Z}[X_1, \ \ldots, \ X_n] \to
		\mathbb{Z}\) that fixes \(\mathbb{Z}\) pointwise, and such that \(((X_1, \
		\ldots, \ X_n)f) \phi = 0\).

		A \textit{system} of equations in integers is a finite set of equations.
		A \textit{solution} to the system is any ring homomorphism that is a
		solution to every equation in the system.
	\end{dfn}

	When we create algorithms that take equations in integers as input, we will
	explicitly state the size of the input.

	\begin{rmk}
		As with group equations, we will usually use a tuple \((x_1, \ \ldots, \
		x_n)\) rather than a ring homomorphism \(\phi \colon \mathbb{Z}[X_1, \
		\ldots, \ X_n] \to \mathbb{Z}\). The homomorphism \(\phi\) can be obtained
		from the tuple by defining \(X_1 \phi = x_i\) for all \(i\), and \(n \phi
		= n\) for all \(n \in \mathbb{Z}\). Since \(\phi\) is a ring homomorphism,
		the action of \(\phi\) on the remainder of \(\mathbb{Z}[X_1, \ \ldots, \
		X_n]\) is now determined.
	\end{rmk}

\subsection{Solution languages}
	Since solutions to group equations are homomorphisms (or tuples of group
	elements), in order to express our sets of solutions as languages we need a
	method of writing our solutions as words. We start by defining a normal
	form.

	\begin{dfn}
		Let \(G\) be a group with a finite generating set \(\Sigma\). A
		\textit{normal form} for \(G\) with respect to \(\Sigma\) is a function
		\(\eta \colon G \to (\Sigma^\pm)^\ast\) that fixes \(\Sigma^\pm\), such that
		\(g \eta\) is a word representing \(g\) for all \(g \in G\).
	\end{dfn}

	Note that as our definition of a normal form uses functions, our normal form
	associates a unique word representative for each group element.

	We now define the solution language to a group equation, with respect to a
	specified normal form.

	\begin{dfn}
		Let \(G\) be a group with a finite generating set \(\Sigma\) and a normal
		form \(\eta\) with respect to \(\Sigma\). Let \(\mathcal E\) be a system
		of equations in \(G\) with variables \(X_1, \ \ldots, \ X_n\). The
		\textit{solution language} of \(\mathcal E\), with respect to \(\Sigma\)
		and \(\eta\), is the language
		\[
			\{(X_1) \phi \eta \# \cdots \# (X_n) \phi \eta \mid \phi \text{ is a
			solution to } \mathcal E\}
		\]
		over the alphabet \(\Sigma^\pm \cup \{\#\}\), where \(\#\) is a symbol not
		in \(\Sigma^\pm\).
	\end{dfn}

	Note that the solution language to an equation in a single variable will not
	require the use of the letter \(\#\), as it is used to separate words
	representing the solutions to individual variables.

	We now define an analogous notion for systems of equations in the ring of
	integers. We pick a letter as a generator, and write the non-negative
	integer \(n\) as this letter to the power of \(n\). For negative integers,
	we introduce an `inverse' of this letter, and express each \(n < 0\) as the
	inverse letter to the power of \(|n|\).

	\begin{dfn}
		Define \(\mu \colon \mathbb{Z} \to \{a\}^\ast \cup \{a^{-1}\}^\ast\) by \(n
		\mu = a^n\).

		Let \(\mathcal E\) be a system of equations in the ring of integers, with
		variables \(X_1, \ \ldots, \ X_n\). The \textit{solution language} to
		\(\mathcal E\) is the language
		\[
			\{(X_1) \phi \mu \# \cdots (X_n) \phi \mu \mid \phi \text{ is a solution
			to } \mathcal E\}
		\]
		over \(\{a, \ a^{-1}, \ \#\}\).
	\end{dfn}

\subsection{EDT0L languages}
	We now define EDT0L languages, which are the class of languages we will use to
	represent solutions. For a more detailed description of EDT0L languages,
	and where they fit in within the collection of languages called L-systems, we
	refer the reader to \cite{math_theory_L_systems}.

	\begin{dfn}
		An \textit{EDT0L system} is a tuple \(\mathcal H = (\Sigma, \ C, \ \omega,
		\ \mathcal{R})\), where
		\begin{enumerate}
			\item \(\Sigma\) is an alphabet, called the \textit{(terminal)
			alphabet};
			\item \(C\) is a finite superset of \(\Sigma\), called the
			\textit{extended alphabet} of \(\mathcal H\);
			\item \(\omega \in C^\ast\) is called the \textit{start word};
			\item \(\mathcal{R}\) is a regular (as a language) set of endomorphisms
			of \(C^\ast\), called the \textit{rational control} of \(\mathcal H\).
		\end{enumerate}
		The language \textit{accepted} by \(\mathcal H\) is \(L(\mathcal H) =
		\{\omega \phi \mid \phi \in \mathcal{R}\} \cap \Sigma^\ast\).

		A language that is accepted by some EDT0L system is called an
		\textit{EDT0L language}.
	\end{dfn}

	We continue with an example of an EDT0L language.

	\begin{ex}
		The language \(L = \{a^{2^n + 1} \mid n \in \mathbb{Z}_{\geq 0}\}\) is
		EDT0L over the alphabet \(\{a\}\). To see this, consider the EDT0L system
		\((\{a\}, \ \{a, \ c\}, \ ac, \ \mathcal R)\) where \(\mathcal R\) is
		defined by the finite-state automaton in Figure \ref{2n_1_EDT0L_fig},
		where \(\phi, \ \theta \in \End(\{a, \ c\}^\ast)\) satisfy
		\begin{align*}
			a \phi & = a^2 & a \theta = a & \\
			c \phi & = c & c \theta = a. &
		\end{align*}
		Alternatively, \(\mathcal R\) can be defined by the rational expression
		\(\phi^\ast \theta\).
		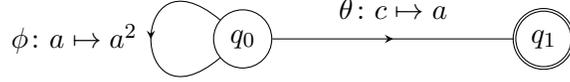
\begin{figure}
			\caption{Rational control for \(L = \{a^{2^n + 1} \mid n \in
			\mathbb{Z}_{\geq 0}\}\), with start state \(q_0\) and accept state
			\(q_1\).}
			\label{2n_1_EDT0L_fig}
			\begin{tikzpicture}
				[scale=.8, auto=left,every node/.style={circle}]
				\tikzset{
				on each segment/.style={
					decorate,
					decoration={
						show path construction,
						moveto code={},
						lineto code={
							\path [#1]
							(\tikzinputsegmentfirst) -- (\tikzinputsegmentlast);
						},
						curveto code={
							\path [#1] (\tikzinputsegmentfirst)
							.. controls
							(\tikzinputsegmentsupporta) and (\tikzinputsegmentsupportb)
							..
							(\tikzinputsegmentlast);
						},
						closepath code={
							\path [#1]
							(\tikzinputsegmentfirst) -- (\tikzinputsegmentlast);
						},
					},
				},
				mid arrow/.style={postaction={decorate,decoration={
							markings,
							mark=at position .5 with {\arrow[#1]{stealth}}
						}}},
			}

				\node[draw] (q0) at (0, 0) {\(q_0\)};
				\node[draw, double] (q1) at (5, 0)  {\(q_1\)};

				\draw[postaction={on each segment={mid arrow}}] (q0) to (q1);

				\draw[postaction={on each segment={mid arrow}}] (q0) to
				[out=140, in=-140, distance=2cm] (q0);

				\node (l1) at (-2.8, 0) {\(\phi \colon a \mapsto a^2\)};

				\node (l2) at (2.5, 0.5) {\(\theta \colon c \mapsto a\)};

			\end{tikzpicture}
		\end{figure}
	\end{ex}

	The class of EDT0L languages is stable under five of the six standard
	operations on languages. The fact that there are non-EDT0L languages which
	are pre-images of EDT0L languages under free monoid homomorphisms means that
	EDT0L languages are not a full algebraic family of languages, like regular,
	context-free or ET0L languages. When dealing with equations, or other
	`parallel' languages, these five operations often prove to be sufficient.

	\begin{lem}[\cite{EDT0L_extensions}, Lemma 2.15]
		\label{EDT0L_closure_properties_lem}
		The class of EDT0L languages is closed under the following operations:
		\begin{enumerate}
			\item Finite unions;
			\item Intersection with regular languages;
			\item Concatenation;
			\item Kleene star closure;
			\item Images under free monoid homomorphisms.
		\end{enumerate}
		Moreover, if the EDT0L systems used in any of these operations can be
		constructed in \(\mathsf{NSPACE}(f)\), for some \(f \colon \mathbb{Z}_{\geq
		0} \to \mathbb{Z}_{\geq 0}\), then there is a computable EDT0L system
		accepting the resultant language that can also be constructed in
		\(\mathsf{NSPACE}(f)\). For images under homomorphisms (5), this requires
		the homomorphism to be able to be written down in \(\ns(f)\).
	\end{lem}


\section{`Dividing' EDT0L languages by a constant}
	\label{division_EDT0L_sec}

    The purpose of this section is to show that given an EDT0L language where
    all words are of the form \(a^i \# b^j\), `dividing' the number \(i\) of
    \(a\)s and the number \(j\) of \(b\)s in a given word by constant values
    \(\gamma\) and \(\delta\) respectively, and removing all words \(a^i \#
    b^j\) where \(i\) is not divisible by \(\gamma\) and \(j\) is not divisible
    by \(\delta\) yields an EDT0L language. We proceed in a similar fashion to
    the arguments used in \cite{EDT0L_extensions}, Section 3, using
    \(\#\)-separated EDT0L systems.

		The concept of \(\#\)-separated EDT0L systems was used in
		\cite{EDT0L_extensions} to show that solution languages to systems of
		equations in direct products of groups where systems of equations have EDT0L
		solution languages are also EDT0L. We use a slightly different definition
		here: we only need a single \(\#\) rather than arbitrarily many, so our
		definition is less general, and we also insist that the start word is of a
		specified form. The latter assumption does not affect the expressive power
		of these systems; preconcatenating the rational control with an appropriate
		endomorphism can convert a \(\#\)-separated system with an arbitrary start
		word into one with a start word of the form we use.

		\begin{dfn}
      Let \(\Sigma\) be an alphabet, and \(\# \in \Sigma\). A
      \textit{\(\#\)-separated EDT0L system} is an EDT0L system \(\mathcal{H}\)
      with: an extended alphabet \(C\), a terminal alphabet \(\Sigma\) and a
      start word \(\omega\) of the form \(\omega = \perp_1 \# \perp_2\) where:
      \(\perp_1, \ \perp_2 \in C \backslash \{\#\}\), and \(c \phi = \#\) if
      and only if \(c = \#\), for every \(c \in C\) and \(\phi\) in the finite
      set of endomorphisms over which the rational control is a regular
      language.
		\end{dfn}

		For space complexity purposes, we will need bounds on the size of extended
		alphabets, and the size of images of letters under endomorphisms in the
		rational control in many of the EDT0L systems we use. We define the term
		\(g\)-bounded to capture this.

		\begin{dfn}
			Let \(\mathcal H = (\Sigma, \ C, \ \perp_1 \# \perp_2, \ \mathcal R)\) be
			a \(\#\)-separated EDT0L system, and let \(g \colon \mathbb{Z}_{\geq 0}
			\to \mathbb{Z}_{\geq 0}\) be a function in terms of a given input size
			\(I\). Let \(B\) be the finite set of endomorphisms of \(C^\ast\) over
			which \(\mathcal R\) is regular. We say that \(\mathcal H\) is
			\textit{\(g\)-bounded} if
			\begin{enumerate}
				\item \(|C| \leq (I)g\);
				\item \(\max\{|c \phi| \mid c \in C, \ \phi \in B\} \leq (I)g\).
			\end{enumerate}
		\end{dfn}

		We will need the fact that the class of languages accepted by
		\(\#\)-separated EDT0L systems is closed under finite unions, with space
		complexity properties being preserved when taking these unions.

		\begin{lem}
			\label{hash_sep_union_lem}
			Let \(L\) and \(M\) be languages over an alphabet \(\Sigma\), accepted by
			\(\#\)-separated EDT0L systems \(\mathcal H\) and \(\mathcal G\), that are
			both \(g\)-bounded and constructible in \(\ns(f)\), for some \(f, \ g
			\colon \mathbb{Z}_{\geq 0} \to \mathbb{Z}_{\geq 0}\).	Then
			\begin{enumerate}
				\item There is a \(\#\)-separated EDT0L system \(\mathcal F\) for \(L
				\cup M\);
				\item The system \(\mathcal F\) is constructible in \(\ns(f)\);
				\item The system \(\mathcal F\) is \((2g + 2)\)-bounded.
			\end{enumerate}
		\end{lem}

		\begin{proof}
			Let \(\mathcal H = (\Sigma, \ C, \ \perp_1 \# \$_1, \ \mathcal R)\) and
			\(\mathcal G = (\Sigma, \ D, \ \perp_2 \# \$_2, \ \mathcal S)\). Let
			\(B_1\) and \(B_2\) be the finite sets of endomorphisms over which
			\(\mathcal R\) and \(\mathcal S\) are regular. We can assume without loss
			of generality that endomorphisms in \(B_1 \cup B_2\) fix elements of
			\(\Sigma\), and also that \(C \backslash \Sigma\) and \(D \backslash
			\Sigma\) are disjoint.

			Let \(\perp\) and \(\$\) be symbols not already used, and let \(E = C \cup
			D \cup \{\perp, \ \$\}\). For each \(\phi \in B_1\), define \(\bar{\phi}\)
			to be the extension of \(\phi\) to \(E\) by \(d \bar{\phi} = d\) for all
			\(d \in E \backslash C\). Similarly extend each \(\phi \in B_2\) to
			\(\bar{\phi} \in \End(E^\ast)\) by \(c \bar{\phi} = c\) for all \(c \in E
			\backslash D\). Define \(\theta_1, \ \theta_2 \in \End(E^\ast)\) by
			\[
				c \theta_1 = \left\{
				\begin{array}{cl}
					\perp_1 & c = \perp \\
					\$_1 & c = \$ \\
					c & \text{otherwise},
				\end{array}
				\right. \qquad
				c \theta_2 = \left\{
				\begin{array}{cl}
					\perp_2 & c = \perp \\
					\$_2 & c = \$ \\
					c & \text{otherwise}.
				\end{array}
				\right. \qquad
			\]
			By construction, \(L \cup M\) is accepted by the \(\#\)-separated
			EDT0L system \(\mathcal F = (\Sigma, \ E, \ \perp \# \$, \ \theta_1
			\mathcal R \cup \theta_2 \mathcal S)\).

			Note that \(\theta_1\) and \(\theta_2\) can both be constructed in
			constant space, and thus the rational control of \(\mathcal F\) is
			constructible in \(\ns(f)\). The start word is constructible in constant
			space. As a union of \(C\) and \(D\) with a constant number of additional
			symbols, \(E\) can be constructed using the same information required to
			construct \(C\) and \(D\), and is thus constructible in \(\ns(f)\).

			We have that \(|E| = |C| + |D| + 2\), and so is bounded by \(2g + 2\). In
			addition, \(B = B_1 \cup B_2 \cup \{\theta_1, \ \theta_2\}\), and so
			\(\max\{|c \phi| \mid c \in E, \ \phi \in B\} = \max(g, \ 2) \leq 2g +
			2\).
		\end{proof}

		We can now prove the central result of this section, about `division' of
		certain EDT0L languages by a constant. To show the space complexity
		properties, we need the EDT0L system we start with to be exponentially
		bounded by the space complexity in which is can be constructed. We will need
		the following notation:

		\begin{ntn}
			Let \(\Sigma\) be an alphabet, \(a \in \Sigma\), and \(w \in
			\Sigma^\ast\). Define \(\#_a(w)\) to be the number of occurrences of the
			letter \(a\) within \(w\).
		\end{ntn}

		The idea of the proof of the following lemma is to index every letter in the
		start word with \(k(|\gamma| - 1)\) \(\cent\)s and \(k\) \(\$\)s, for some
		\(k \in \mathbb{Z}_{\geq 0}\), and ensure this fact is preserved under the
		action of the rational control (possibly changing the value of \(k\)). We
		then add a new endomorphism to map \(\cent\)-indexed letters to
		\(\varepsilon\), and \(\$\)-indexed letters to \(a\) (or \(a^{-1}\) if
		\(\gamma < 0\)).

		\begin{lem}
			\label{EDT0L_division_lem}
			Let \(X \subseteq \mathbb{Z}_{\geq 0}^2\) be such that for each \(x, \ y
			\in \mathbb{Z}_{\geq 0}\) there is at most one \(x' \in \mathbb{Z}_{\geq
			0}\) such that \((x, \ x') \in X\), and at most one \(y' \in
			\mathbb{Z}_{\geq 0}\) with \((y', \ y) \in X\). Let \(\gamma, \ \zeta \in
			\mathbb{Z}\) be non-zero. Let
			\[
				L = \{a^x \# b^y \mid (x, \ y) \in X\}.
			\]
			\begin{enumerate}
				\item If \(L\) is EDT0L, then so is the language
				\[
					L_{\gamma, \zeta} = \{a^{\frac{x}{\gamma}} \# b^{\frac{y}{\zeta}} \mid
					(x, \ y) \in X, \ \gamma | x, \ \zeta | y\};
				\]
				\item If \(L\) is accepted by a \(\#\)-separated EDT0L system \(\mathcal
				H = (\Sigma, \ C, \ \perp_1 \# \perp_2, \ \mathcal R)\) that is
				\(\exp(f)\)-bounded and constructible in \(\ns(f)\), where \(f \colon
				\mathbb{Z}_{\geq 0} \to \mathbb{Z}_{\geq 0}\) is at least linear, then
				\(L_{\gamma, \zeta}\) is accepted by an EDT0L system that is
				constructible in \(\ns(f gh)\), where \(g\) is linear in \(|\gamma|\),
				and \(h\) is linear in \(|\zeta|\).
			\end{enumerate}
		\end{lem}

		\begin{proof}

			We will use \(\mathcal H\) to define an EDT0L system for
			\[
				M = \{a^{\frac{x}{|\gamma|}} \# b^y \mid (x, \ y) \in X, \
				\gamma | x\}.
			\]
			Firstly note that if \(|\gamma| = 1\), then \(M = L\), and thus \(M\) is
			accepted by \(\mathcal H\), which satisfies the conditions in (2). So
			assume \(|\gamma| \geq 2\). Let \(\cent\) and \(\$\) be symbols not
			already used. Let \(\hat{c}^{\nu}\) be a distinct copy of \(c\) for each
			\(c \in C\) and \(\nu \in \{\cent, \ \$\}^\ast\). Let \(C^\text{ind} =
			\{\hat{c}^{\nu} \mid \nu \in \{\cent, \ \$\}^\ast, \ |\nu| \leq |\gamma|\}
			\sqcup C \sqcup \{F\}\), where \(F\) is a new symbol. We will use \(F\) as
			a `fail symbol'.

 			Let \(B \subseteq \End(C^\ast)\) be the finite set over which \(\mathcal
 			R\) is a regular language. For each \(\phi \in B\), the finite set
 			\(\Phi_\phi \subseteq \End((C^\text{ind})^\ast)\) is defined as follows. If  \(\nu \in \{\cent,
 			\ \$\}^\ast\) satisfies \(|\nu| \leq |\gamma|\), and \(c \in C\) is such
 			that \(c \phi = d_1 \cdots d_n\), with \(n \geq 1\), \(d_1, \ \ldots, \
 			d_n \in C\) (in particular, \(c \phi \neq \varepsilon\)), then let \(\psi \in
 			\End((C^\text{ind})^\ast)\) be defined by
			\[
				\hat{c}^\nu \psi = \hat{d}_1^{\alpha_1} \cdots \hat{d}_n^{\alpha_n},
			\]
			for some \(\alpha_1, \ \ldots, \ \alpha_n \in \{\cent, \ \$\}^\ast\) such
			that \(|\alpha_i| \leq |\gamma|\) for all \(i\), and one of the following
			holds:
			\begin{enumerate}
				\item \(\#_\$(\alpha_1 \cdots \alpha_n) = \#_\$(\nu)\), and
				\(\#_{\cent}(\alpha_1 \cdots \alpha_n) = \#_{\cent}(\nu)\);
				\item \(\#_\$(\alpha_1 \cdots \alpha_n) = \#_\$(\nu) + 1\), and
				\(\#_{\cent}(\alpha_1 \cdots \alpha_n) = \#_{\cent}(\nu) + |\gamma| -
				1\).
			\end{enumerate}
			If \(c \phi = \varepsilon\), then
			\[
				\hat{c}^{\nu} \psi = \left\{
				\begin{array}{cl}
					F & \nu \neq \varepsilon \\
					\varepsilon & \nu = \varepsilon.
				\end{array}
				\right.
			\]
			In addition, \(\psi\) fixes \(F\), and acts the same way as \(\phi\) on
			letters in \(C\).	We take \(\Phi_\phi\) to be the set of all such
			\(\psi\), as \(\alpha_1, \ \ldots, \ \alpha_n\) vary for each \(c \in C\),
			satisfying the stated conditions. Let \(\bar{\mathcal R}\) be the regular
			set of endomorphisms defined by replacing each occurrence of \(\phi\)
			within \(\mathcal R\) with \(\Phi_\phi\). Now define \(\theta \in
			\End((C^\text{ind})^\ast)\) by
			\[
				\hat{c}^{\nu} \theta = \left\{
				\begin{array}{cl}
					c & c \in \Sigma, \ \nu = \$ \\
					\varepsilon & c \in \Sigma, \ \nu = \cent \\
					F & \text{otherwise},
				\end{array}
				\right.
				\qquad
				F \theta = F, \qquad c \theta = c \text{ for all } c \in C.
			\]
			Let \(\mathcal G = (\Sigma, \ C^\text{ind}, \ \hat{\perp}_1 \# \perp_2, \
			\bar{\mathcal R} \theta)\). By construction, any word in \(\hat{\perp}_1
			\bar{\mathcal R}\) either contains an \(F\), or is a word in \(\perp_1
			\mathcal{ R}\) with hats on letters and indices that concatenate to form a
			word \(\nu \in \{\cent, \ \$\}^\ast\) of length \(n |\gamma|\) for some
			\(n \in \mathbb{Z}_{\geq 0}\), with \(\#_{\cent}(\nu) = n (|\gamma| -
			1)\), and \(\#_\$(\nu) = n\). Thus the set of words in \(\hat{\perp}_1
			\bar{\mathcal R} \theta \cap \Sigma\) equals \(\{a^{\frac{x}{|\gamma|}}
			\mid (x, \ y) \in S, \ \gamma | x\}\). It follows that \(\mathcal G\)
			accepts \(M\).

			We now consider the space complexity in which \(\mathcal G\) can be built.
			Firstly, note that to output \(C^\text{ind}\) we simply need to output
			\((2^{|\gamma| + 1} - 1)\) (the number of words of length at most
			\(|\gamma|\) over a two letter alphabet) additional copies of \(C\), plus
			the letter \(F\). Doing this simply requires us to track the copy we're
			on, and since \(\log(2^{|\gamma| + 1} - 1)\) is linear in \(|\gamma|\),
			this can be done in \(\ns(f)\). The start word can be output in constant
			space.

			We now consider the rational control. To construct \(\bar{\mathcal R}\),
			we need to follow the process to construct \(\mathcal R\), except we need
			to construct \(\Phi_\phi\) whenever the finite-state automaton for
			\(\mathcal R\) constructs \(\phi\). Let \(\psi \in \Phi_\phi\), and note
			that if \(c \in C\) and \(\nu \in \{\cent, \ \$\}^\ast\) is such that
			\(|\nu| \leq |\gamma|\), then there are at most
			\[
				(\max\{|c \varphi| \mid c \in C, \ \varphi \in B\})^{|\nu|}
				\leq (\max\{|c \varphi| \mid c \in C, \ \varphi \in B\})^{|\gamma|}
			\]
			possible values that \(\hat{c}^{\nu} \psi\) can take. As a result,
			\[
				|\Phi_\phi| \leq (\max\{|c \varphi| \mid c \in C, \ \varphi \in
				B\})^{|\gamma|} \cdot |C| \cdot (2^{|\gamma| + 1} - 1).
			\]
			Thus
			\[
				\log|\Phi_\phi| \leq |\gamma| \log (\max\{|c \varphi| \mid c \in C, \
				\varphi \in B\}) + \log|C| + (|\gamma| + 1) \log 2.
			\]
			To construct \(\Phi_\phi\), we simply need to store the information
			required to construct \(\phi\), together with a counter to tell us how
			many \(\psi\) in \(\Phi_\phi\) we have already constructed. Since
			\(\log|\Phi_\phi|\) is bounded by \(fg\) for some linear function \(g\) in
			\(|\gamma|\), we can construct \(\Phi_\phi\), and hence \(\bar{\mathcal
			R}\) in \(\ns(fg)\). As \(\theta\) can be constructed in constant space,
			it follows that the rational control, and hence \(\mathcal G\), can be
			constructed in \(\ns(fg)\).

			To see that the language accepted by \(\mathcal G\) is in fact \(M\),
			first note that for any \(\bar{\phi} \in \bar{\mathcal R}\),
			\(\hat{\perp}_1 \bar{\phi}\) will be obtained from a word \(\perp_1
			\phi\), for some \(\phi \in \mathcal R\) by attaching \(k(|\gamma| - 1)\)
			\(\cent\) indices and \(k\) \(\$\) indices, for some \(k \in
			\mathbb{Z}_{\geq 0}\). This will only be accepted if \((\perp_1 \#
			\perp_2) \phi \in \Sigma^\ast\), and every letter in \(\perp_1 \phi\) has
			precisely one index on it. In such a case, \(|\perp_1 \phi| = k|\gamma|\)
			(in fact \(\perp_1 \phi = a^{\pm k |\gamma|}\)), and precisely \(k\) of
			these letters will be indexed by \(\$\), the rest being indexed by
			\(\cent\). Hitting such a word with \(\theta\) will delete all letters
			indexed with a single \(\cent\), and map the \(\$\)-indexed \(a\)s to
			\(a\) and \(\$\)-indexed \(a^{-1}\)s to \(a^{-1}\), leaving the word
			\(a^{\pm x} \# b^y\) to be accepted. Thus \(M\) is accepted by \(\mathcal
			G\).

			We now show that
			\[
				N = \{a^{\frac{x}{\gamma}} \# b^y \mid (x, \ y) \in X, \
			 \gamma | x\}
			\]
			is accepted by an EDT0L system, constructible in \(\ns(fg)\). Note that if
			\(\gamma \geq 0\), then \(M = N\), and there is nothing to prove.
			Otherwise, \(\gamma < 0\). Define \(\pi \in \End((C^\text{ind})^\ast)\) by
			\(a \pi = a^{-1}\), \(a^{-1} \pi = a\) and all other letters are fixed by
			\(\pi\). Then \((\Sigma, \ C^{\text{ind}}, \ \perp_1 \# \perp_2, \
			\bar{\mathcal R} \theta \pi)\) accepts \(N\), as we have just flipped the
			sign of the \(a\)s in \(M\). Moreover, as \(\mathcal G\) is constructible
			in \(\ns(fg)\), so is our system for \(N\). In addition, the stated bounds
			on the size of the extended alphabet and the images of endomorphisms of
			\(\mathcal G\) hold for our system for \(N\) as well.

			To obtain an EDT0L system for \(\{a^{\frac{x}{\gamma}} \#
			b^{\frac{y}{\zeta}} \mid (x, \ y) \in X, \ \gamma | x, \ \zeta | y\}\), we
			simply apply the same method we used to obtain \(N\) from \(L\) to \(N\),
			except modifying \(\perp_2\) and \(b\), rather than \(\perp_1\) and \(a\).
		\end{proof}

\section{Pell's equation}
	\label{pell_eqn_sec}

	The purpose of this section is to study solutions to Pell's equation, which
	eventually allows us to show that the solution language to a
	quadratic equation in the ring of integers is EDT0L.

	We start with a lemma that shows languages that arise as part of recursively
	defined integer sequences with non-negative integer coefficients are EDT0L.
	We will later show that solutions to Pell's equation are of this form.

	\begin{lem}
		\label{gen_fibonacci_EDT0L_lem}
		Let \((p_n)_{n \geq 0}\), \((q_n)_{n \geq 0}\) and \((r_n)_{n \geq 0}\) be
		integer sequences, defined recursively by a relation
		\[
			p_n = \alpha_1 p_{n - 1} + \alpha_2 q_{n - 1} + \alpha_3 r_{n - 1}, \quad
			q_n = \beta_1 p_{n - 1} + \beta_2 q_{n - 1} + \beta_3 r_{n - 1}, \quad
			r_n = \gamma_1 p_{n - 1} + \gamma_2 q_{n - 1} + \gamma_3 r_{n - 1}
		\]
		where \(\alpha_1, \ \alpha_2, \ \alpha_3, \ \beta_1, \ \beta_2, \ \beta_3, \
		\gamma_1, \ \gamma_2, \ \gamma_3 \in \mathbb{Z}_{\geq 0}\). Suppose also
		that \(p_0, \ q_0, \ r_0 \in \mathbb{Z}_{\geq 0}\) or \(p_0, \ q_0, \ r_0
		\in \mathbb{Z}_{\leq 0}\). Then
		\begin{enumerate}
			\item The language \(L = \{a^{p_n} \mid n \in \mathbb{Z}_{\geq
			0}\}\) is EDT0L;
			\item Taking the input size to be \(I = \max(\alpha_1, \ \alpha_2, \
			\alpha_3, \ \beta_1, \ \beta_2, \ \beta_3, \ \gamma_1, \ \gamma_2, \
			\gamma_3, \ p_0, \ q_0, \ r_0)\), an EDT0L system \(\mathcal H\) for \(L\)
			is constructible in non-deterministic logarithmic space;
			\item The system \(H\) is \(f\)-bounded for some linear function \(f\);
			\item The rational control of \(\mathcal H\) is of the form \(\theta
			\varphi^\ast \psi\), and \(\perp \theta \varphi^n \psi = a^{p_n}\),
			where \(\perp\) is the start word of \(\mathcal H\).
		\end{enumerate}
	\end{lem}

	\begin{proof}
		We will define an EDT0L system to accept \(L\). Let \(\Sigma = \{a, \
		a^{-1}\}\). Our extended alphabet will be \(C = \Sigma
		\cup \{a_p, \ a_p^{-1}, \ a_q, \ a_q^{-1}, \ a_r, \ a_r^{-1}, \ \perp\}\),
		and our start word will be \(\perp\). Define \(\theta \in \End(C^\ast)\) by
		\[
			c \theta = \left\{
			\begin{array}{cl}
				a_p^{p_0} a_q^{q_0} a_r^{r_0} & c = \perp \\
				c & \text{otherwise}.
			\end{array}
			\right.
		\]
		Define \(\varphi \in \End(C^\ast)\) by
		\begin{align*}
			a_p^{\pm 1} \varphi & = a_p^{\pm \alpha_1} a_q^{\pm \beta_1} a_r^{\pm \gamma_1} \\
			a_q^{\pm 1} \varphi & = a_p^{\pm \alpha_2} a_q^{\pm \beta_2} a_r^{\pm \gamma_2} \\
			a_r^{\pm 1} \varphi & = a_p^{\pm \alpha_3} a_q^{\pm \beta_3} a_r^{\pm \gamma_3}
		\end{align*}
		and fix all other letters. Finally, define \(\psi \in \End(C^\ast)\) by
		\begin{align*}
			a_p^{\pm 1} \psi & = a^{\pm 1} \\
			a_q^{\pm 1} \psi & = a_r^{\pm_1} \psi = \varepsilon,
		\end{align*}
		and all other letters are fixed.
		Our rational control will be \(\theta \varphi^\ast \psi\).

		First note that \(u = \perp \theta \varphi^n\) contains either \(a_p\) or
		\(a_p^{-1}\), but not both, and the same holds for \(a_q\) and \(a_q^{-1}\),
		and \(a_r\) and \(a_r^{-1}\). So we can abuse notation and take the
		definition of \(\#_{a_p}\) when applied to such a word to be \(\#_{a_p}(u)\)
		if \(u\) contains an \(a_p\), \(-\#_{a_p^{-1}}(u)\) if \(u\) contains an
		\(a_p^{-1}\), and \(0\) if it contains neither. We similarly abuse notation
		with \(\#_{a_q}\) and \(\#_{a_r}\).

		We will show by induction that \(u = \perp \theta \varphi^n\) satisfies
		\(\#_{a_q}(u) = p_n\), \(\#_{a_q}(u) = q_n\), and \(\#_{a_r}(u) = r_n\).
		This holds by definition for \(n = 0\). Inductively suppose it is true for
		some \(k - 1\). Then \(\perp \theta \varphi^k = u\), for some \(u \in \{a_p,
		\ a_p^{-1}, \ a_q, \ a_q^{-1}, \ a_r, \ a_r^{-1}\}^\ast\), with
		\(\#_{a_q}(u) = p_{k - 1}\), \(\#_{a_q}(u) = q_{k - 1}\), and \(\#_{a_r}(u)
		= r_{k - 1}\). Using the definition of \(\varphi\), and our inductive
		hypothesis we have
		\begin{align*}
		& \#_{a_p}(u \varphi) = \alpha_1 \#_{a_p}(u) + \alpha_2 \#_{a_q}(u)
		+ \alpha_3 \#_{a_r} (u) = \alpha_1 p_{k - 1} + \alpha_2 q_{k - 1}
		+ \alpha_3 r_{k - 1} = p_k \\
		& \#_{a_q}(u \varphi) = \beta_1 \#_{a_p}(u) + \beta_2 \#_{a_q}(u)
		+ \beta_3 \#_{a_r} (u) = \beta_1 p_{k - 1} + \beta_2 q_{k - 1}
		+ \beta_3 r_{k - 1} = q_k \\
		& \#_{a_r}(u \varphi) = \gamma_1 \#_{a_p}(u) + \gamma_2 \#_{a_q}(u)
		+ \gamma_3 \#_{a_r} (u) = \gamma_1 p_{k - 1} + \gamma_2 q_{k - 1}
		+ \gamma_3 r_{k - 1} = r_k.
		\end{align*}
		It now follows that \(\perp \theta \varphi^n \psi = a^{p_n}\), and thus (1)
		and (4) are true.

	 	We now show that the EDT0L system \((\Sigma, \ C, \ \perp, \ \theta
	 	\varphi^\ast \psi)\) is constructible in non-deterministic linear space.
	 	Writing down \(\Sigma\), \(C\), \(\psi\) and the start word can be done in
	 	constant space. Writing down \(\theta\) can be done by remembering \(p_0\),
	 	\(q_0\) and \(r_0\), and thus can be done in non-deterministic logarithmic
	 	space, since storing an integer \(r\) requires \(\log(r)\) plus a constant
	 	bits. It remains to show that \(\varphi\) can be defined in
	 	non-deterministic logarithmic space. To write down \(\varphi\), we simply
	 	need to know the coefficients \(\alpha_i\), \(\beta_i\) and \(\gamma_i\)
	 	for \(i \in \{1, \ 2, \ 3\}\). Since these can all be stored using \(\log
	 	\alpha_i\), \(\log \beta_i\) and \(\log \gamma_i\) bits, respectively plus
	 	constants, (2) follows.

		Finally note that \(|C| = 8\), which is constant. In addition, \(|c
		\varphi|\), for \(c \in C\), is bounded by a linear function of the values
		\(\alpha_i\), \(\beta_i\) and \(\gamma_i\), \(|c \theta| \leq p_0 + q_0 +
		r_0\), and \(|c \psi| \leq 1\). We have now shown (3).
	\end{proof}

	To show that the solution language to a general quadratic equation in two
	variables is EDT0L, we follow Lagrange's method to reduce it to the
	generalised Pell's equation, and then to Pell's equation. This reduction is
	detailed in \cite{new_look_old_equation}. We start with the definition of
	Pell's equation.

	\begin{dfn}
		\textit{Pell's equation} is the equation \(X^2 - DY^2 = 1\) in the ring of
		the integers, where \(X\) and \(Y\) are variables, and \(D \in \mathbb{Z}_{>
		0}\) is not a perfect square. The \textit{fundamental solution} to Pell's
		equation \(X^2 - DY^2 = 1\) is the minimal (with respect to the \(\ell^1\)
		metric on \(\mathbb{Z}^2\)) non-negative integer solution that is not \((1,
		\ 0)\).
	\end{dfn}

	The solutions to Pell's equation have long been understood. The following
	lemma details one of several ways of constructing them.

	\begin{lem}[\cite{quad_dioph_eqns_book}, Theorem 3.2.1]
		\label{Pell_eqn_sequence_lem}
			There are infinitely many solutions to Pell's equation \(X^2 - DY^2 = 1\),
			and these are \(\{(x_n, \ y_n) \mid n \in \mathbb{Z}_{\geq 0}\}\), where
			\((x_0, \ y_0) = (1, \ 0)\), and \((x_n, \ y_n)\) is recursively defined
			by
			\[
				x_n = x_1 x_{n - 1} + Dy_1 y_{n - 1}, \qquad y_n = y_1 x_{n - 1} + x_1
				y_{n - 1},
			\]
			where \((x_1, \ y_1)\) is the fundamental solution.
	\end{lem}

	We give an explicit example of Pell's equation and its solutions.

	\begin{ex}
		Consider Pell's equation \(X^2 - 2Y^2 = 1\). It is not hard to check using
		brute force that the fundamental solution is \((3, \ 2)\) (although there
		are more efficient methods of doing this: see for example
		\cite{quad_dioph_eqns_book}). Thus by Lemma \ref{Pell_eqn_sequence_lem}, we
		can construct the set of all solutions using the sequence \((x_n, \ y_n)
		\subseteq \mathbb{Z}^2\), defined recursively by \((x_0, \ y_0) = (1, \
		0)\), and
		\[
			x_n = 3 x_{n - 1} + 4 y_{n - 1}, \qquad y_n = 2 x_{n - 1} + 3
			y_{n - 1}.
		\]
		At this point, we could just apply Lemma \ref{gen_fibonacci_EDT0L_lem}  and
		Lemma \ref{EDT0L_closure_properties_lem} to show that the language \(\{a^x
		\# a^y \mid (x, \ y) \in \mathbb{Z}_{\geq 0}^2 \text{ is a solution to } X^2
		- 2Y^2 = 1\}\) is EDT0L, however we will explicitly construct an EDT0L
		system. Our extended alphabet will be \(C = \{a_x, \ \bar{a}_x, \ \ a_y, \
		\bar{a}_y, \ a, \ \#\}\) and our start word will be \(a_x \# \bar{a}_x\).
		Let \(\varphi \in End(C^\ast)\) be defined by
		\begin{align*}
			a_x \varphi & = a_x^3 \bar{a}_y^2
			& \bar{a}_x \varphi & = \bar{a}_x^3 a_y^2 & \\
			a_y \varphi & = \bar{a}_x^4 a_y^3
			& \bar{a}_y \varphi & = a_x^4 \bar{a}_y^3, & \\
			a \varphi & = a & \# \varphi & = \#.
		\end{align*}
		Define \(\theta \in \End(C^\ast)\) by
		\begin{align*}
			a_x \theta & = a_y \theta = a \theta = a \\
			\bar{a}_x \theta & = \bar{a}_y \theta = \varepsilon \\
			\# \theta & = \#.
		\end{align*}
		Our rational control will be \(\varphi^\ast \theta\) (alternatively, see
		Figure \ref{Pell_eqn_ex_fig}).
		\begin{figure}
			\caption{Rational control for \(L = \{a^x \# a^y \mid (x, \ y) \in
			\mathbb{Z}_{\geq 0}^2 \text{ is a solution to } X^2 - 2Y^2 = 1\}\), with
			start state \(q_0\) and accept state \(q_1\).}
			\label{Pell_eqn_ex_fig}
			\begin{tikzpicture}
				[scale=.8, auto=left,every node/.style={circle}]
				\tikzset{
				on each segment/.style={
					decorate,
					decoration={
						show path construction,
						moveto code={},
						lineto code={
							\path [#1]
							(\tikzinputsegmentfirst) -- (\tikzinputsegmentlast);
						},
						curveto code={
							\path [#1] (\tikzinputsegmentfirst)
							.. controls
							(\tikzinputsegmentsupporta) and (\tikzinputsegmentsupportb)
							..
							(\tikzinputsegmentlast);
						},
						closepath code={
							\path [#1]
							(\tikzinputsegmentfirst) -- (\tikzinputsegmentlast);
						},
					},
				},
				mid arrow/.style={postaction={decorate,decoration={
							markings,
							mark=at position .5 with {\arrow[#1]{stealth}}
						}}},
			}

				\node[draw] (q0) at (0, 0) {\(q_0\)};
				\node[draw, double] (q1) at (5, 0)  {\(q_1\)};

				\draw[postaction={on each segment={mid arrow}}] (q0) to (q1);

				\draw[postaction={on each segment={mid arrow}}] (q0) to
				[out=140, in=-140, distance=2cm] (q0);

				\node (l1) at (-1.9, 0) {\(\varphi\)};

				\node (l2) at (2.4, 0.4) {\(\theta\)};

			\end{tikzpicture}
		\end{figure}
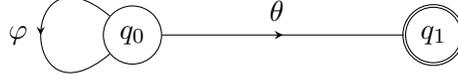
		By construction, \(\#_{a_x}(a_x \# \bar{a}_x \varphi^n) = \#_{\bar{a}_x}(a_x
		\# \bar{a}_x \varphi^n) = x_n\) and \(\#_{a_y}(a_x \# \bar{a}_x \varphi^n) =
		\#_{\bar{a}_y}(a_x \# \bar{a}_x \varphi^n) = y_n\), and thus \(a_x \#
		\bar{a}_x \varphi^n \theta = a^{x_n} \# a^{y_n}\).
	\end{ex}

	In addition to the recursive structure of all solutions, we need a bound on
	the size of the fundamental solution. This allows us to give a bound on the
	space complexity in which the EDT0L system can be constructed.

	\begin{lem}[\cite{lentra_pell_eqn_bound}, Section 3]
		\label{pell_eqn_bound_lem}
		Let \((x_1, \ y_1)\) be the fundamental solution to Pell's equation
		\(X^2 - DY^2 = 1\). Then
		\[
			\log(x_1 + y_1 \sqrt{D}) < \sqrt{D} (\log(4D) + 2).
		\]
	\end{lem}

	Understanding solutions to arbitrary two-variable quadratic equations using
	Lagrange's method requires us to have an understanding of the images of the
	solutions to Pell's equation under linear functions: that is \(\alpha x +
	\beta y + \gamma\) for constant \(\alpha, \ \beta, \ \gamma \in \mathbb{Z}\),
	where \((x, \ y)\) is a solution. If \(\alpha\), \(\beta\) and \(\gamma\) are
	either all non-negative or all non-positive, this corresponds to concatenating
	EDT0L languages in parallel, which is not too difficult using standard EDT0L
	constructions.

	On the other hand, if the signs of these three integers are not all the same,
	more work needs to be done. This occurs because we represent the integer \(n
	\in \mathbb{Z}\) by \(a^n\), where \(a\) is a letter. Thus if we want to `add'
	\(-3\) and \(5\), this corresponds in language terms to trying to concatenate
	\(a^{-3}\) and \(a^5\), which results in \(a^{-3} a^5\), which is not equal
	(as a word) to \(a^2\). One cannot, in general, freely reduce all words in an
	EDT0L language to form an EDT0L language. There are in fact cases where such a
	reduction will result in a language that is not recursive; that is a language
	which is not accepted by a Turing machine, or whose complement is not accepted
	by a Turing machine.

	To tackle the harder cases presented to us by `subtraction', we instead study
	the integer sequences themselves, and show they satisfy recurrence relations
	that can be used to define EDT0L systems.

	\begin{lem}
		\label{Pell_seq_single_rec_relation_lem}
		Let \((x_n), \ (y_n) \subseteq \mathbb{Z}_{\geq 0}\) be sequences of
		solutions to Pell's equation \(X^2 - DY^2 = 1\). Let \(\alpha, \ \beta \in
		\mathbb{Z}_{\geq 0}\). Let \((z_n) \subseteq \mathbb{Z}\) be the sequence
		defined by \(z_n = \alpha x_n - \beta y_n\). Then, for all \(n \in
		\mathbb{Z}_{\geq 2}\)
		\begin{enumerate}
			\item \(x_n = 2 x_1 x_{n - 1} - x_{n - 2}\);
			\item \(y_n = 2 x_1 y_{n - 1} - y_{n - 2}\);
			\item \(z_n = 2 x_1 z_{n - 1} - z_{n - 2}\).
		\end{enumerate}
	\end{lem}

	\begin{proof}
		We will proceed by induction on \(n\) to show (1) and (2). First note that
		\[
			2 x_1 y_1 - y_0 = x_1 y_1 + y_1 x_1 = y_2.
		\]
		Additionally,
		\[
			2x_1 x_1 - x_0 = x_1^2 + (x_1^2 - 1) = x_1^2 + D y_1^2 = x_2.
		\]
		Thus (1) and (2) hold when \(n = 2\). Suppose the result holds when \(n = k\).
		Then
		\begin{align*}
			x_{k + 1} & = x_1 x_k + D y_1 y_k \\
			& = x_1 (2 x_1 x_{k - 1} - x_{k - 2}) + D y_1 (2 x_1 y_{k - 1} - y_{k -
			2}) \\
			& = 2 x_1 (x_1 x_{k - 1} + D y_1 y_{k - 1}) - (x_1 x_{k - 2} + D y_1 y_{k
			- 2}) \\
			& = 2 x_1 x_k - x_{k - 1}. \\
			y_{k + 1} & = y_1 x_k + x_1 y_k \\
			& = y_1 (2 x_1 x_{k - 1} - x_{k - 2}) + x_1 (2 x_1 y_{k - 1} - y_{k - 2})
			\\
			& = 2 x_1 (y_1 x_{k - 1} + x_1 y_{k - 1}) - (y_1 x_{k - 2} + x_1 y_{k -
			2}) \\
			& = 2 x_1 y_k - y_{k - 1}.
		\end{align*}
		It remains to show (3). We have, using (1) and (2),
		\begin{align*}
			z_n & = \alpha x_n - \beta y_n \\
			& = \alpha (2 x_1 x_{n - 1} - x_{n - 2}) - \beta (2 x_1 y_{n - 1} - y_{n -
			2}) \\
			& = 2 x_1 (\alpha x_{n - 1} - \beta y_{n - 1}) - (\alpha x_{n - 2} - \beta
			y_{n - 2}) \\
			& = 2 x_1 z_{n - 1} - z_{n - 2}. \qedhere
		\end{align*}
	\end{proof}

	Using Lemma \ref{Pell_seq_single_rec_relation_lem}, we can now prove some
	results about the sequence \((z_n)\) that show that it is indeed a type of
	sequence as mentioned by Lemma \ref{gen_fibonacci_EDT0L_lem}.

	\begin{lem}
		\label{alphax_nbeta_y_n_lem}
		Let \((x_n), \ (y_n) \subseteq \mathbb{Z}_{\geq 0}\) be sequences of
		solutions to Pell's equation \(X^2 - DY^2 = 1\). Let \(\alpha, \ \beta \in
		\mathbb{Z}_{\geq 0}\). Let \((z_n) \subseteq \mathbb{Z}\) be the sequence
		defined by \(z_n = \alpha x_n - \beta y_n\). Then
		\begin{enumerate}
			\item If \(N = \left\lceil \log_2 \frac{\alpha}{\beta} \right \rceil\),
			then \((z_n)_{n \geq N} \subseteq \mathbb{Z}_{\geq 0}\)
			or \((z_n)_{n \geq N} \subseteq \mathbb{Z}_{< 0}\);
			\item The sequence \((w_n)_{n \geq 1} \subseteq \mathbb{Z}\)
		  defined by \(w_n = z_n - z_{n - 1}\) satisfies for
			all \(n \in \mathbb{Z}_{\geq 2}\),
			\[
				z_n = (2 x_1 - 1) z_{n - 1} + w_{n - 1}, \qquad
				w_n = (2 x_1 - 2) z_{n - 1} + w_{n - 1};
			\]
			\item If \((z_n)_{n \geq N}\) is a sequence of non-negative integers then
			\((w_n)_{n \geq N}\) is, and if \((z_n)_{n \geq N}\) is a sequence of
			non-positive integers then \((w_n)_{n \geq N}\) is;
			\item The sequence \((w_n)_{n \geq N}\) is monotone;
			\item If \(\gamma \in \mathbb{Z}\) and \(M = \left \lceil \log_2
			\frac{(\gamma + 3)\alpha}{\beta} \right \rceil\), then \((z_n + \gamma)_{n
			\geq M}, \ (w_n + \gamma)_{n \geq M} \subseteq \mathbb{Z}_{\geq 0}\) or
			\((z_n + \gamma)_{n \geq M}, \ (w_n + \gamma)_{n \geq M} \subseteq
			\mathbb{Z}_{\leq 0}\).
		\end{enumerate}
	\end{lem}

	\begin{proof}
		We start by showing (1). Let \(\gamma = \frac{\beta}{\sqrt{D}}\). Then, if
		\(n \in \mathbb{Z}_{\geq 0}\),
		\begin{align*}
			z_n = \alpha x_n - \beta y_n = \alpha x_n - \gamma \sqrt{D} y_n.
		\end{align*}
		We have that \(z_n \geq 0\) if and only if \(z_n(\gamma x_n + \alpha \sqrt{D}
		y_n) \geq 0\). Note that
		\begin{align*}
			z_n (\gamma x_n + \alpha \sqrt{D} y_n)
			& = (\alpha x_n - \gamma \sqrt{D} y_n) (\gamma x_n + \alpha \sqrt{D} y_n)
			\\
			& = \alpha \gamma x_n^2 + \alpha^2 \sqrt{D} x_n y_n - \gamma^2 \sqrt{D}
			x_n y_n - \alpha \gamma Dy_n^2 \\
			& = \alpha \gamma (x_n^2 - D y_n^2) + \sqrt{D} x_n y_n (\alpha^2 -
			\gamma^2) \\
			& = \alpha \gamma + \sqrt{D} x_n y_n (\alpha^2 - \gamma^2).
		\end{align*}
		If \(\alpha \geq \gamma\), the above expression must be at least \(0\), so
		\(z_n \geq 0\) for all \(n \in \mathbb{Z}_{\geq 0}\), and there is nothing
		to prove. Otherwise, suppose \(\gamma > \alpha\), and write \(\gamma =
		\alpha + \delta\) for some \(\delta > 0\). Then
		\begin{align*}
			z_n (\gamma x_n + \alpha \sqrt{D} y_n)
			& = \alpha \gamma + \sqrt{D} x_n y_n (\alpha^2 - \gamma^2) \\
			& = \alpha (\alpha + \delta) + \sqrt{D} x_n y_n (\alpha^2 - (\alpha +
			\delta)^2) \\
			& = \alpha^2 + \alpha \delta - \sqrt{D} x_n y_n (\delta^2 + 2 \alpha
			\delta).
		\end{align*}
		It follows that \(z_n < 0\) if and only if \(\alpha^2 + \alpha \delta
		- \sqrt{D} x_n y_n (\delta^2 + 2 \alpha \delta) < 0\). That is,
		\[
			x_n y_n > \frac{\alpha^2 + \alpha \delta}{\sqrt{D} (\delta^2 + 2 \alpha
			\delta)}.
		\]
		Noting that \(x_n\) and \(y_n\) are both strictly increasing, and if \(n
		\geq 1\), \(x_n y_n > 1\), it suffices to find \(N \in \mathbb{Z}_{>0}\)
		such that if \(n = N\) the above inequality holds. By
		\cref{Pell_eqn_sequence_lem}, we have that \(x_n \geq x_1 x_{n - 1}\) and
		\(y_n \geq x_1 y_{n - 1}\), and so \(x_n y_n \geq x_1^{2n - 1} y_1\). Noting
		that \(x_1 \geq 2\) and \(y_1 \geq 1\), it follows that \(x_n y_n \geq
		2^n\). Note that \(\frac{\alpha}{\beta} = \frac{\alpha \gamma}{\sqrt{D}} =
		\frac{\alpha^2 + \alpha \delta}{\sqrt{D}} \geq \frac{\alpha^2 + \alpha
		\delta}{\sqrt{D}( \delta^2 + 2 \alpha \delta)}\), and so choosing \(N =
		\left\lceil \log_2 \frac{\alpha}{\beta} \right \rceil\) will satisfy the
		stated conditions.

		For (2), let \(n \in \mathbb{Z}_{\geq 2}\). Then, using
		\cref{Pell_seq_single_rec_relation_lem},
		\begin{align*}
			z_n & = 2 x_1 z_{n - 1} - z_{n - 2} = (2 x_1 - 1) z_{n - 1} + w_{n - 1}.
			\\
			w_n & = z_n - z_{n - 1} = 2 x_1 z_{n - 1} - z_{n - 2} - z_{n - 1}
			= (2 x_1 - 2) z_{n - 1} + w_{n - 1}.
		\end{align*}

		We now show (3). As with our proof of (1), let \(\gamma =
		\frac{\beta}{\sqrt{D}}\). Then, for all \(n \in \mathbb{Z}_{> 0}\),
		\[
			w_n = z_n - z_{n - 1} = \alpha x_n - \gamma \sqrt{D} y_n - \alpha x_{n - 1}
			+ \gamma \sqrt{D} y_{n - 1} = \alpha (x_n - x_{n - 1}) - \gamma \sqrt{D}
			(y_n - y_{n - 1}).
		\]

		Since \(x_n\) and \(y_n\) are both strictly increasing, \(w_n \geq 0\) if
		and only if \(w_n (\gamma (x_n - x_{n - 1}) + \alpha \sqrt{D} (y_n - y_{n -
		1})) \geq 0\). Let \(u_n = w_n (\gamma (x_n - x_{n - 1}) + \alpha \sqrt{D}
		(y_n - y_{n - 1}))\). Write \(v_n = (\alpha^2 - \gamma^2) \sqrt{D} (x_n -
		x_{n - 1}) (y_n - y_{n - 1})\). We have
		\begin{align*}
			u_n	& = (\alpha (x_n - x_{n - 1}) - \gamma \sqrt{D} (y_n - y_{n - 1}))
			(\gamma (x_n - x_{n - 1}) + \alpha \sqrt{D} (y_n - y_{n - 1})) \\
			& = \alpha \gamma ((x_n - x_{n - 1})^2 - D(y_n - y_{n - 1})^2) + (\alpha^2
			- \gamma^2) \sqrt{D} (x_n - x_{n - 1}) (y_n - y_{n - 1}) \\
			& = \alpha \gamma (x_n^2 - 2x_n x_{n - 1} + x_{n - 1}^2 - D y_n^2 + 2D y_n
			y_{n - 1} - D y_{n - 1}^2) + v_n \\
			& = \alpha \gamma ((x_n^2 - D y_n^2) + (x_{n - 1}^2 - D y_{n - 1}^2) + 2D
			y_n y_{n - 1} - 2x_n x_{n - 1}) + v_n \\
			& = 2 \alpha \gamma (1 + D y_n y_{n - 1} - x_n x_{n - 1}) + v_n \\
			& = 2 \alpha \gamma (1 + D (y_1 x_{n - 1} + x_1 y_{n - 1}) y_{n - 1} -
			(x_1 x_{n - 1} +  D y_1 y_{n - 1}) x_{n - 1}) + v_n \\
			& = 2 \alpha \gamma (1 + D y_1 x_{n - 1} y_{n - 1} + Dx_1 y_{n - 1}^2 -
			x_1 x_{n - 1}^2 -  D y_1 y_{n - 1} x_{n - 1}) + v_n \\
			& = 2 \alpha \gamma (1 + Dx_1 y_{n - 1}^2 - x_1 x_{n - 1}^2) + v_n \\
			& = 2 \alpha \gamma (1 - x_1(x_{n - 1}^2 - D y_{n - 1}^2)) + v_n \\
			& = 2 \alpha \gamma (1 - 1) + v_n \\
			& = (\alpha^2 - \gamma^2) \sqrt{D} (x_n - x_{n - 1}) (y_n - y_{n - 1}).
		\end{align*}
		Note that \((\alpha^2 - \gamma^2) \sqrt{D} (x_n - x_{n -
		1}) (y_n - y_{n - 1}) \geq 0\) if and only if \(\alpha \geq \gamma\); that
		is \(\frac{\alpha}{\beta} \sqrt{D} \geq 1\). As we saw in the proof of part
		(1), \(\frac{\alpha}{\beta} \sqrt{D} \geq 1\) implies \((z_n)_{n \geq N}\)
		is a sequence of non-negative integers, and \(\frac{\alpha}{\beta} \sqrt{D}
		< 1\) implies \((z_n)_{n \geq N}\) is a sequence of non-positive integers,
		as required.

		For (4), we show \((w_n)_{n \geq N}\) is monotone. First note that
		\((w_n)_{n \geq N}\) and \((z_n)_{n \geq N}\) are both sequences of
		non-negative integers or sequences of non-positive integers. In addition,
		\(w_n = w_{n - 1} + 2x_1 z_{n - 1}\) for all \(n \in \mathbb{Z}_{> 0}\). So
		if \(n \in \mathbb{Z}_{\geq N}\), then \(|w_n| = |w_{n - 1}| + |2x_1 z_{n -
		1}| \geq |w_{n - 1}|\). As \((w_n)_{n \geq N}\) is a sequence of
		non-negative integers or a sequence of non-positive integers, it must be
		monotone.

		We finally consider (5). It suffices to show that \(|z_M| > |\gamma|\) and
		\(|w_M| \geq |\gamma|\), then together with the fact that \(M \geq N\), and
		using the fact that \((z_n)_{n \geq N}\) is monotone by (3), and \(w_n\) is
		monotone by (4), we have that \((z_n)_{n \geq M}\) and \((w_n)_{n \geq M}\)
		are both sequences of non-positive or non-negative integers. We know that
		\(|z_{n + N}| > 2^{n - 1}\) and \(w_n \geq 2^{n - 2}\) using (2), together
		with the fact that \(x_1 > 1\), and so \(2 x_1 - 1 > 2\), so taking any \(M
		\geq N + \log_2(|\gamma| + 2)\) suffices. As \(N = \left
		\lceil \log_2 \frac{\alpha}{\beta} \right \rceil\), taking \(M =
		\left \lceil \log_2 \frac{(|\gamma| + 3)\alpha}{\beta} \right
		\rceil\), as per the statement of the lemma, satisfies the desired
		condition.
	\end{proof}

	Before we apply Lemma \ref{gen_fibonacci_EDT0L_lem} to show that some of
	these solution languages are EDT0L, we need to add constants to the
	differences of multiples of solutions.

	\begin{lem}
		\label{alphax_nbeta_y_n_gamma_lem}
		Let \((x_n), \ (y_n) \subseteq \mathbb{Z}_{\geq 0}\) be sequences of
		solutions to Pell's equation \(X^2 - DY^2 = 1\). Let \(\alpha, \ \beta \in
		\mathbb{Z}_{\geq 0}\) and \(\gamma \in \mathbb{Z}\). Let \((z_n), \ (t_n)
 		\subseteq \mathbb{Z}\) be sequences defined by \(z_n = \alpha x_n - \beta
		y_n\) and \(t_n = z_n + \gamma\). Then
		\begin{enumerate}
			\item The sequence \((s_n)_{n \geq 1} \subseteq \mathbb{Z}\)
			defined by \(s_n = z_n - z_{n - 1} + \gamma\) satisfies for
			all \(n \in \mathbb{Z}_{\geq 2}\),
			\[
				t_n = (2 x_1 - 1) z_{n - 1} + s_{n - 1}, \qquad
				s_n = (2 x_1 - 2) z_{n - 1} + s_{n - 1};
			\]
			\item If \(\gamma \in \mathbb{Z}\) and \(M = \left \lceil \log_2
			\frac{(\gamma + 2)\alpha}{\beta} \right \rceil\), then \((t_n)_{n
			\geq M} \subseteq \mathbb{Z}_{\geq 0}\) or \((t_n)_{n \geq M}
			\subseteq \mathbb{Z}_{\leq 0}\);
			\item If \((t_n)_{n \geq M}\) is a sequence of non-negative integers then
			\((s_n)_{n \geq M}\) and \((z_n)_{n \geq M}\) are, and if \((t_n)_{n \geq
			M}\) is a sequence of non-positive integers then \((s_n)_{n \geq M}\) and
			\((z_n)_{n \geq M}\) are;
		\end{enumerate}
	\end{lem}

	\begin{proof}
		We start with (1). Let \(w_n = z_n - z_{n - 1}\), for all \(n \in
		\mathbb{Z}_{> 0}\). If \(n \in \mathbb{Z}_{> 0}\), then using
		\cref{alphax_nbeta_y_n_lem}
		\begin{align*}
			t_n & = z_n + \gamma \\
			& = (2 x_1 - 1) z_{n - 1} + w_{n - 1} + \gamma \\
			& = (2 x_1 - 1) z_{n - 1} + s_{n - 1}, \\
			s_n & = w_n + \gamma \\
			& = (2 x_1 - 2) z_{n - 1} + w_{n - 1} + \gamma \\
			& = (2 x_1 - 2) z_{n - 1} + s_{n - 1}.
		\end{align*}
		Parts (2) and (3) follow by \cref{alphax_nbeta_y_n_lem} (5).
	\end{proof}

	To allow us to show space complexity properties, we need bounds of many of
	the integers we have introduced.

	\begin{lem}
		\label{x_N_bound_lem}
		Let \(S\) be the set of all  non-negative solutions (as ordered pairs) to
		Pell's equation \(X^2 - DY^2 = 1\). Let \(\alpha, \ \beta \in \mathbb{Z}_{>
		0}\) and \(\gamma \in \mathbb{Z}\), and \(M = \max \left(2, \ \left \lceil
		\log_2 \frac{(|\gamma| + 3)\alpha}{\beta} \right \rceil\right)\). Let \(z_n
		= \alpha x_n - \beta y_n\) and \(t_n = z_n + \gamma\) for all \(n \in
		\mathbb{Z}_{\geq 0}\). Let \(w_n = z_n - z_{n - 1}\) and \(s_n = w_n +
		\gamma\) for all \(n \in \mathbb{Z}_{> 0}\).

		Then there is a function \(f\) that is logarithmic in
		\(\alpha\), \(\beta\) and \(|\gamma|\), and a function \(g\) that is linear
		in \(D\), such that \(\log(x_M)\), \(\log(y_M)\), \(\log|z_M|\),
		\(\log|w_M|\), \(\log|t_M|\) and \(\log|s_M|\) are all bounded by \(fg\).
	\end{lem}

	\begin{proof}
		\cref{Pell_seq_single_rec_relation_lem}, together with the fact that
		\((x_n)\) is strictly increasing, implies that \(x_n \leq (2x_1)^n\) for all
		\(n \geq 1\). Thus \(x_M \leq (2x_1)^M = (2 x_1)^{\left\lceil \log_2
		\frac{(|\gamma| + 3)\alpha}{\beta} \right \rceil} = \left \lceil
		\frac{(|\gamma| + 2)\alpha}{\beta} \right \rceil x_1^{\left\lceil \log_2
		\frac{(|\gamma| + 2)\alpha}{\beta} \right \rceil}\). Using
		\cref{pell_eqn_bound_lem}, we have that
		\begin{align*}
			\log(x_M)
			& \leq \log\left\lceil\frac{(|\gamma| + 3)\alpha}{\beta} \right \rceil +
			\left\lceil \log_2 \frac{(|\gamma| + 3)\alpha}{\beta} \right \rceil \log(x_1) \\
			& \leq \log(\alpha + 1) + \log(|\gamma| + 4) + \log(\beta + 1) +  (\log(\alpha + 1) +
			\log(|\gamma| + 4) + \log(\beta + 1))\sqrt{D} (\log 4D + 2) \\
			& \leq (\log(\alpha + 1) + \log(|\gamma| + 4) + \log(\beta + 1))(2 + 3D).
		\end{align*}
		Since \(y_M < x_M\), we have that \(\log(y_M)\) is also bounded by
		\((\log(\alpha + 1)  + \log(|\gamma| + 4) + \log(\beta + 1))(2 + 3D)\).

		Let \(n \geq 1\). Then, using the fact that \(\alpha x_n\) and \(\beta (y_n
		+ 1)\) are both at least \(1\), we have
		\begin{align*}
			\log|z_n| & = \log|\alpha x_n - \beta y_n| \\
			& = \log(\alpha) + \log(x_n) + \log(\beta) + \log(y_n + 1) \\
			& \leq \log(\alpha) + \log(x_n) + \log(\beta) + \log(y_n) + 1.
		\end{align*}
		Using the fact that \(x_M\) and \(y_M\) are bounded by \((\log(\alpha + 1)
		+ \log(|\gamma| + 4) + \log(\beta + 1))(2 + 3D)\), we now have that \(z_M
		\leq 2(\log(\alpha + 1)  + \log(|\gamma| + 4) + \log(\beta + 1))(2 + 3D) +
		\log(\alpha) + \log(\beta) + 1\).

		We have that \(w_M = z_M - z_{M - 1}\). Noting that \(M - 1 \geq 1\),
		\(x_{M - 1} \leq x_M\) and \(y_{M - 1} \leq y_M\), it follows that
		\begin{align*}
			\log|w_M| & = \log|z_M - z_{M - 1}|\\
			& \leq \log|z_M| + \log|z_{M - 1}| \\
			& \leq 2 \log(\alpha) + 2 \log(\beta) + 4 (\log(\alpha + 1)  +
			\log(|\gamma| + 4) + \log(\beta + 1))(2 + 3D) (2 + 3D) + 2.
		\end{align*}
		Since \(t_M = z_M + \gamma\) and \(s_M = w_M + \gamma\), we have
		that \(t_M\) and \(s_M\) are bounded by the same expressions as \(z_M\)
		and \(w_M\) if \(\gamma = 0\). Otherwise,
		\begin{align*}
			\log|t_M| & \leq \log|z_M| + \log|\gamma| \\
			& \leq 2(\log(\alpha + 1)  + \log(|\gamma| + 4) + \log(\beta + 1))(2 + 3D)
			+ \log(\alpha) + \log(\beta) + 1 + \log(|\gamma|),\\
			\log|s_M| & \leq \log|w_M| + \log|\gamma| \\
			& \leq \log(\alpha) + 2 \log(\beta) + 4 (\log(\alpha + 1)  +
			\log(|\gamma| + 4) + \log(\beta + 1))(2 + 3D) + 2 + \log|\gamma|.
		\end{align*}
		Taking \(f = 4 (\log(\alpha + 1)  + \log(|\gamma| + 4) + \log(\beta + 1)) +
		2\) and \(g = 2 + 3D\), the result follows.
	\end{proof}

	We have now completed the set up to show that the solution language to
	Pell's equation is always EDT0L. More than that, we can show that applying
	linear functions to the variables will still give this outcome. We need the
	bounds on the size of our extended alphabet and images of endomorphisms so
	that we can apply Lemma \ref{EDT0L_division_lem} later on.

	\begin{lem}
		\label{Pell_multiple_EDT0L_lem}
		Let \(S\) be the set of all  non-negative solutions (as tuples) to Pell's
		equation \(X^2 - DY^2 = 1\), and \(\alpha, \ \beta, \ \gamma, \ \delta, \
		\epsilon, \ \zeta \in \mathbb{Z}\). Then
		\begin{enumerate}
			\item The language \(L = \{a^{\alpha x + \beta y + \gamma} \# b^{\delta x
			+ \epsilon y + \zeta} \mid (x, \ y) \in S\}\) is EDT0L;
			\item A \(\#\)-separated EDT0L system \(\mathcal H\) for \(L\) is
			constructible in \(\ns(fg)\), where \(f\) is logarithmic in
			\(\max(|\alpha|, \ |\beta|, \ |\gamma|, \ |\delta|, \ |\epsilon|, \
			|\zeta|)\), and \(g\) is linear in \(D\);
			\item The system \(\mathcal H\) is \(h_1 h_2\)-bounded, where \(h_1\) is
			linear in \(\max(|\alpha|, \ |\beta|, \ |\gamma|, \ |\delta|, \
			|\epsilon|, \ |\zeta|)\), and \(h_2\) is exponential in \(D\).
		\end{enumerate}
	\end{lem}

	\begin{proof}
		Let \(z_n = \alpha x_n + \beta y_n\) and \(t_n = z_n + \gamma\) for \(n \in
		\mathbb{Z}_{\geq 0}\), and \(s_n = z_n - z_{n - 1} + \gamma\) for \(n \in
		\mathbb{Z}_{> 0}\). Let \(M_\gamma = \max \left(2, \ \left \lceil \log_2
		\frac{(|\gamma| + 3)\alpha}{\beta} \right \rceil\right)\), \(M_\zeta = \max
		\left(2, \ \left \lceil \log_2 \frac{(|\zeta| + 3)\alpha}{\beta} \right
		\rceil\right)\), and \(M = \max(M_\gamma, \ M_\zeta)\). We will first
		construct an EDT0L system for
		\[
			K = \{a^{t_n} \mid n \in \mathbb{Z}_{\geq M)}\}.
		\]
		If \(\alpha \leq 0\) and \(\beta \geq 0\), or \(\alpha \geq 0\) and \(\beta
		\leq 0\), Lemma \ref{alphax_nbeta_y_n_gamma_lem} tells us that the sequences
		\((t_n)_{n \geq M}\), \((z_n)_{n \geq M}\) and \((s_n)_{n \geq M}\) satisfy
		the conditions of Lemma \ref{gen_fibonacci_EDT0L_lem}, and thus \(K\) is
		accepted by an EDT0L system \(\mathcal H = (\{a, \ a^{-1}\}, \ C, \ \perp, \
		\theta \varphi^\ast \psi)\).

		If \(\alpha\) and \(\beta\) are both non-negative or non-positive, then
		\[
			z_n = \alpha (x_1 x_{n - 1} + D y_1 y_{n - 1}) + \beta (y_1 x_{n - 1}
			+ x_1 y_{n - 1})
			= (\alpha x_1 + \beta y_1) x_{n - 1} + (\alpha D y_1 + \beta x_1) y_{n - 1}.
		\]
		This, together with the recurrence relations in Lemma
		\ref{Pell_eqn_sequence_lem}, gives that \((z_n)_{n \geq M}\), \((x_n)_{n
		\geq M}\) and \((y_n)_{n \geq M}\) satisfy the conditions of Lemma
		\ref{gen_fibonacci_EDT0L_lem}, and so in we also have in this case that
		\(L\) is accepted by an EDT0L system \(\mathcal H = (\{a, \ a^{-1}\}, \ C, \
		\perp, \ \theta \varphi^\ast \psi)\).

		We next consider the space complexity in which \(\mathcal H\) can be
		constructed. By Lemma \ref{pell_eqn_bound_lem}, \(\log(x_1)\) and
		\(\log(y_1)\) are both bounded by \(2 + 3D\). By Lemma \ref{x_N_bound_lem},
		\(\log(x_M)\), \(\log(y_M)\), \(\log|z_M|\), \(\log|t_M|\) and \(\log|s_M|\)
		are all bounded by \(fg\), where \(f\) is logarithmic in \(|\alpha|\),
		\(|\beta|\), \(|\gamma|\) and \(|\zeta|\), and \(g\) is linear in \(D\).
		Thus we can use Lemma \ref{gen_fibonacci_EDT0L_lem} (2) to say that
		\(\mathcal H\) is constructible in \(\ns(fg)\). We also know from Lemma
		\ref{gen_fibonacci_EDT0L_lem} (4), that \(|C|\) and \(\max\{|c \phi| \mid c
		\in C, \ \phi \in \{\psi, \ \theta, \ \varphi\}\}\) are both bounded in
		terms of an exponential function of \(fg\). Thus \(|C|\) and \(\max\{c \phi
		\mid c \in C, \ \phi \in \{\psi, \ \theta, \ \varphi\}\}\) are both bounded
		by \(h_1 h_2\) where \(h_1\) is linear in \(|\alpha|\), \(|\beta|\),
		\(|\gamma|\) and \(|\zeta|\), and \(h_2\) is exponential in \(D\).

		Let \(\hat{z}_n = \delta x + \epsilon y\) and \(\hat{t}_n = \hat{z}_n +
		\zeta\) for \(n \in \mathbb{Z}_{\geq 0}\), and \(\hat{s}_n = \hat{z}_n -
		\hat{z}_{n - 1} + \zeta\) for \(n \in \mathbb{Z}_{> 0}\). With the same
		arguments we used to show \(K\) is accepted by \(\mathcal H\), we have that
		\[
			\{b^{\hat{t}_n} \mid n \in \mathbb{Z}_{\geq M}\},
		\]
		is accepted by an EDT0L system \(\hat{\mathcal H} = (\{b, \ b^{-1}\}, \ D, \
		\$, \ \sigma \rho^\ast \tau)\). In addition, \(\hat{\mathcal H}\) is
		constructible in \(\ns(\hat{f} \hat{g})\), and \(|D|\) and \(\max\{|c \phi|
		\mid c \in D, \ \phi \in \{\sigma, \ \rho, \ \tau\}\}\) are both bounded by
		\(\hat{h}_1 \hat{h}_2\), where \(\hat{f}\) and \(\hat{h}_1\) are logarithmic
		and linear respectively in \(|\delta|\), \(|\epsilon|\), \(|\gamma|\) and
		\(|\zeta|\), and \(\hat{g}\) and \(\hat{h}_2\) are linear and exponential
		respectively in \(D\). Redefining \(f\), \(g\), \(h_1\) and \(h_2\) to be
		the sum of themselves and their hatted versions, gives that both \(\mathcal
		H\) and \(\hat{\mathcal H}\) are constructible in \(\ns(fg)\). In addition,
		\(|C|\), \(|D|\), \(\max\{|c \phi| \mid c \in C, \ \phi \in \{\psi, \
		\theta, \ \varphi\}\}\) and  \(\max\{|c \phi| \mid c \in D, \ \phi \in
		\{\sigma, \ \rho, \ \tau\}\}\) are all bounded by \(h_1 h_2\).

		Without loss of generality, we can assume that \(C\) and \(D\) are disjoint,
		and \(\# \notin C \cup D\). For each endomorphism \(\phi \in \{\psi, \
		\theta, \ \varphi, \ \sigma, \ \rho, \ \tau\}\), let \(\bar{\phi} \in
		\End{(C \cup D \cup \{\#\})^\ast}\) be defined to be the extension of
		\(\phi\) to \(C \cup D \cup \{\#\}\) which acts as the identity on wherever
		it was not previously defined on. It follows that
		\[
			P = \{a^{t_n} \# b^{\hat{t}_n} \mid n \in \mathbb{Z}_{\geq M}\}
		\]
		is accepted by the \(\#\)-separated EDT0L system \(\mathcal G = (\{a, \
		a^{-1}, \ b, \ b^{-1}, \ \#\}, \ C \cup D \cup \{\#\}, \ \perp \# \$, \
		\theta \sigma (\varphi \rho)^\ast \psi \tau)\).

		Since \(\mathcal H\) and \(\hat{\mathcal H}\) are constructible in
		\(\ns(fg)\), so is \(\mathcal G\). In addition, \(|C \cup D \cup \{\#\}|\)
		is bounded by \(h_1 h_2 + 1\), and \(\max\{|c\bar{\phi}| \mid c \in C \cup D
		\cup \{\#\}, \ \phi \in \{\psi, \ \theta, \ \varphi, \ \sigma, \ \rho, \
		\tau\}\}\) is bounded by \(h_1h_2\). Redefining \(h_2\) to be \(h_2 + 1\),
		gives that \(\mathcal G\) satisfies all of the conditions of the lemma.

		We now consider the language
		\[
			Q = \{a^{t_n} \# b^{\hat{t}_n} \mid n \in \{0, \ \ldots, \ M - 1\}\}
		\]
		Note that using Lemma \ref{hash_sep_union_lem}, it now suffices to show that
		\(Q\) is accepted by a \(\#\)-separated EDT0L system that is constructible
		in \(\ns(fg)\), and whose extended alphabet and images of letters under
		endomorphisms in the alphabet of the rational control are bounded by
		\(h_1 h_2\).

		Let \(E = \{\perp_1, \ \perp_2, \ a, \ a^{-1}, \ b, \ b^{-1}, \ \#\}\). We
		will use \(E\) as our extended alphabet, and \(\perp_1 \# \perp_2\) as our
		start symbol. For each \(n \in \{0, \ \ldots, \ M - 1\}\), define \(\pi_n
		\in \End(E^\ast)\) by
		\[
			c \pi_n = \left\{
			\begin{array}{cl}
				a^{t_n} & c = \perp_1 \\
				b^{\hat{t}_n} & c = \perp_2 \\
				c & \text{ otherwise}.
			\end{array}
			\right.
		\]
		It follows that \(Q\) is accepted by the \(\#\)-separated EDT0L system
		\[
			\mathcal F = (\{a, \ a^{-1}, \ b, \ b^{-1}, \ \#\}, \ E, \ \perp_1 \#
			\perp_2, \ \{\pi_0, \ \ldots, \ \pi_{M - 1}\}).
		\]
		Note that \(t_0 = \alpha + \gamma\), \(\hat{t}_0 = \delta + \zeta\). Thus
		\(\log|t_0|\) and \(\log|\hat{t_0}|\) are both bounded by a logarithmic
		function \(f_1\) in terms of \(|\alpha|\), \(|\beta|\), \(|\gamma|\),
		\(|\delta|\), \(|\epsilon|\) and \(|\zeta|\). By redefining \(f\) to be \(f
		+ f_1\), we have that \(\log|t_0|\) and \(\log|\hat{t}_0|\) are bounded by
		\(fg\). In addition, \(\log|t_M|\) and \(\log|\hat{t}_M|\) are both bounded
		by \(fg\). Since \((t_n)\) and \((\hat{t}_n)\) are monotone, and terms are
		effectively computable by Lemma \ref{alphax_nbeta_y_n_gamma_lem}, each
		\(\pi_n\) can be constructed in \(\ns(fg)\). As \(E\) and \(\perp_1 \#
		\perp_2\) are constructible in constant space, it follows that \(\mathcal
		F\) is also constructible in \(\ns(fg)\).

		We have that \(|t_0|\) and \(|\hat{t}_0|\) are bounded by a linear function
		\(f_3\) in terms of \(|\alpha|\), \(|\beta|\), \(|\gamma|\), \(|\delta|\),
		\(|\epsilon|\) and \(|\zeta|\). By redefining \(h_1\) to be \(h_1 + f_3\),
		we have that \(|t_0|\), \(|\hat{t}_0|\), \(|t_M|\) and \(|\hat{t}_M|\) are
		all bounded by \(h_1 h_2\), and thus \(\max\{ |c \pi_i| \mid c \in E, \ i
		\in \{0, \ \ldots, \ M - 1\}\}\) and \(|E|\) are both bounded by \(h_1
		h_2\).
	\end{proof}

\section{Quadratic equations in the ring of integers}
	\label{quad_eqns_sec}

	Having completed the work on Pell's equation, we now consider more general
	quadratic equations in the ring of integers, working up to an arbitrary
	two-variable equation. Our main goal is to show that the solution language to
	an arbitrary two-variable quadratic equation is EDT0L, with an
	EDT0L system that is constructible in non-deterministic polynomial space. We
	start with the general Pell's equation.

	\begin{dfn}
		A \textit{general Pell's equation} is an equation \(X^2 - DY^2 = N\) in the
		ring of integers, where \(X\) and \(Y\) are variables, \(N \in \mathbb{Z}
		\backslash \{0\}\) and \(D \in \mathbb{Z}_{> 0}\) is not a perfect square.

		A non-negative integer solution \((x, \ y)\) to the general Pell's equation
		\(X^2 - DY^2 = N\) is called \textit{primitive} if \(\gcd(x, \ y) = 1\).
	\end{dfn}

	Before we can generalise Lemma \ref{Pell_multiple_EDT0L_lem} to a general
	Pell's equation, we first generalise it to the primitive solutions to a
	general Pell's equation. The following result allows us to construct
	primitive solutions to a general Pell's equation from the solutions to the
	corresponding Pell's equation, and a given primitive solution.

	\begin{lem}[\cite{quad_dioph_eqns_book}, Section 4.1]
		\label{gen_pell_eqn_seq_lem}
		Let \((x_0, \ y_0)\) be a primitive solution to the general Pell's equation
		\(X^2 - DY^2 = N\). Let \((u_n, \ v_n)\) be the sequence of solutions (as
		described in Lemma \ref{Pell_eqn_sequence_lem}) to \(U^2 - D V^2 = 1\).
		Define \(((x_n, \ y_n))_n \subseteq \mathbb{Z}_{\geq 0}^2\) by
		\[
			x_n = x_0 u_n + Dy_0 v_n, \qquad y_n = y_0 u_n + x_0 v_n.
		\]
		Then \((x_n, \ y_n)\) is a primitive solution to \(X^2 - DY^2 = N\) for all
		\(n \in \mathbb{Z}_{\geq 0}\).
	\end{lem}

	We will put an equivalence relation on the set of primitive solutions to a
	general Pell's equation. This will allow us to consider one class at a time,
	then use Lemma \ref{hash_sep_union_lem} to take the union.

	\begin{dfn}
		Let \((x, \ y)\) and \((x', \ y')\) be primitive solutions to the general
		Pell's equation \(X^2 - DY^2 = N\). If there exists a primitive solution
		\((x_0, \ y_0)\) such that \((x, \ y) = (x_m, \ y_m)\) and \((x', \ y') =
		(x_n, \ y_n)\), for some \(m, \ n \in \mathbb{Z}_{\geq 0}\) (using the
		construction in Lemma \ref{gen_pell_eqn_seq_lem}), we say \((x, \ y)\) and
		\((x', \ y')\) are \textit{associated} with each other.
	\end{dfn}

	\begin{lem}
		Association of primitive solutions to a general Pell's equation is an
		equivalence relation.
	\end{lem}

	Equivalence classes of primitive solutions, which we will call classes,
	have a notion of a fundamental solution, similar to the fundamental solution
	to Pell's equation.

	\begin{dfn}
		The \textit{class} of a primitive solution \((x, \ y)\) of a general
		Pell's equation is the equivalence class of all primitive solutions
		associated with \((x, \ y)\).

		The \textit{fundamental solution} of a class of primitive solutions to
		a general Pell's equation is the minimal element of the class.
	\end{dfn}

	We will need the following bounds for the space complexity results.

	\begin{lem}[\cite{quad_dioph_eqns_book}, Theorem 4.1.1 and Theorem 4.12]
		\label{gen_pell_eqn_bound_lem}
		Let \((x_0, \ y_0)\) be the fundamental solution of a class of primitive
		solutions to the general Pell's equation \(X^2 - DY^2 = N\). Let \((u_1, \
		v_1)\) be the fundamental solution to \(X^2 - DY^2 = 1\). If \(N > 0\), then
		\[
			0 \leq x_0 \leq \sqrt{\frac{N(u_1 + 1)}{2}}, \qquad
			0 < y_0  \leq \frac{v_1 \sqrt{N}}{\sqrt{2(u_1 + 1)}}.
		\]
		If \(N < 0\), then
		\[
			0 \leq x_0 \leq \sqrt{\frac{|N|(u_1 - 1)}{2}}, \qquad
			0 < y_0 \leq \frac{v_1 \sqrt{|N|}}{\sqrt{2(u_1 - 1)}}.
		\]
	\end{lem}

	Since the size of fundamental solutions to a general Pell's equation is
	bounded, there can only be finitely many, and hence only finitely many
	classes.

	\begin{lem}
		\label{finitely_many_classes_lem}
		There are finitely many classes of primitive solutions to a general Pell's
		equation.
	\end{lem}

	We now show that the results stated in Lemma \ref{Pell_multiple_EDT0L_lem}
	hold for primitive solutions to a general Pell's equation. We use the
	characterisation in Lemma \ref{gen_pell_eqn_seq_lem} to reduce the problem to
	Pell's equation, and then apply Lemma \ref{Pell_multiple_EDT0L_lem}.

	\begin{lem}
		\label{gen_pell_eqn_primitive_EDT0L_lem}
		Let \(S\) be the set of primitive solutions to the general Pell's equation
		\(X^2 - DY^2 = N\), and \(\alpha, \ \beta, \ \gamma, \ \delta, \ \epsilon, \
		\zeta \in \mathbb{Z}\). Then
		\begin{enumerate}
			\item The language \(L = \{a^{\alpha x + \beta y + \gamma} \# b^{\delta x
			+ \epsilon y + \zeta} \mid (x, \ y) \in S\}\) is EDT0L;
			\item A \(\#\)-separated EDT0L system \(\mathcal H\) for \(L\) is
			constructible in \(\ns(fg)\), where \(f\) is logarithmic in
			\(\max(|\alpha|, \ |\beta|, \ |\gamma|, \ |\delta|, \ |\epsilon|, \
			|\zeta|, \ |N|, \ D)\), and \(g\) is linear in \(D\);
			\item The system \(\mathcal H\) is \(h_1 h_2\)-bounded, where \(h_1\) is
			linear in \(\max(|\alpha|, \ |\beta|, \ |\gamma|, \ |\delta|, \
			|\epsilon|, \ |\zeta|, \ |N|)\), and \(h_2\) is exponential in \(D\).
		\end{enumerate}
	\end{lem}

	\begin{proof}
		Since finite unions of EDT0L languages are EDT0L, and the properties in (2)
		and (3) are preserved (Lemma \ref{hash_sep_union_lem}), using Lemma
		\ref{finitely_many_classes_lem} it is sufficient to show that for any class
		of primitive solutions \(K\), the language
		\[
			M = \{a^{\alpha x + \beta y + \gamma} \# b^{\delta x + \epsilon y + \zeta}
			\mid (x, \ y) \in K\}
		\]
		is accepted by an EDT0L system that satisfies the conditions (2) and (3).
		Let \(((u_n, \ v_n))_n\) be the sequence of non-negative integer solutions
		to \(X^2 - DY^2 = 1\). Let \((x_0, \ y_0)\) be the fundamental solution in
		\(K\). Then we can write elements of \(K\) as \((x_n, \ y_n)\), where
		\[
			x_n = x_0 u_n + D y_0 v_n, \qquad y_n = y_0 u_n + x_0 v_n,
		\]
		for some \(n \in \mathbb{Z}_{\geq 0}\). For any \(n \in \mathbb{Z}_{\geq
		0}\),
		\begin{align*}
			\alpha x_n + \beta y_n + \gamma
			& = \alpha (x_0 u_n + D y_0 v_n) + \beta (y_0 u_n + x_0 v_n) + \gamma \\
			& = (\alpha x_0 + \beta y_0) u_n + (\alpha D y_0 + \beta x_0) v_n +
			\gamma, \\
			\delta x_n + \epsilon y_n + \zeta
			& = \delta (x_0 u_n + D y_0 v_n) + \epsilon (y_0 u_n + x_0 v_n) + \zeta\\
			& = (\delta x_0 + \epsilon y_0) u_n + (\delta D y_0 + \epsilon x_0) v_n
			+ \zeta \\
		\end{align*}
		Note that, by \cref{gen_pell_eqn_bound_lem} and \cref{pell_eqn_bound_lem}
		\begin{align*}
			\log(x_0) & = \frac{1}{2} \log \left(\frac{N(u_1 + 1)}{2} \right) \leq
			\log(N) + \log(u_1 + 1) \leq 2 + 3D + \log(N), \\
			\log(y_0) & \leq \log(v_1 \sqrt{N}) \leq 2 + 3D + \log(N).
		\end{align*}
		Thus
		\begin{align*}
			\log|\delta x_0 + \epsilon y_0| & \leq \log(|\delta| (x_0 + 1)) +
			\log(|\epsilon|(y_0 + 1)) \\
			& \leq \log|\delta| + \log|\epsilon| + \log(x_0) + \log(y_0) + 2 \\
			& \leq \log|\delta| + \log|\epsilon| + 2 + 6D + 2 \log(N), \\
			\log|\delta D y_0 + \epsilon x_0| & \leq \log(|\delta| D (y_0 + 1))
			+ \log(|\epsilon| (x_0 + 1)) \\
			& \leq \log|\delta| + \log|\epsilon| + \log(D) + \log(x_0) + \log(y_0) +
			2.
		\end{align*}
		Note that the above inequalities also hold with \(\delta\) replaced by
		\(\alpha\), and \(\epsilon\) replaced by \(\beta\). The result now follows
		from Lemma \ref{Pell_multiple_EDT0L_lem}.
	\end{proof}

	We now consider all solutions to a general Pell's equation. We start with a
	reduction from a non-primitive solution to a primitive solution.

	\begin{lem}
		\label{Pell_non_primitive_solns_lem}
		Let \((x, \ y) \in \mathbb{Z}_{\geq 0}^2\), and let \(k = \gcd(x, \ y)\).
		Then \((x, \ y)\) is a solution to the general Pell's equation \(X^2 - DY^2
		= N\) if and only if \(k^2 | N\), and \(\left(\frac{x}{k}, \ \frac{y}{k}
		\right)\) is a primitive solution to the general Pell's equation \(X^2 -
		DY^2 = \frac{N}{k^2}\).
	\end{lem}

	\begin{proof}
		We have \(x^2 - Dy^2 = N\) if and only if
		\[
			\frac{N}{k^2} = \frac{x^2 - Dy^2}{k^2} = \left(\frac{x}{k}\right)^2
			- D \left( \frac{y}{k} \right)^2. \qedhere
		\]
	\end{proof}

	It is now possible to generalise Lemma \ref{gen_pell_eqn_primitive_EDT0L_lem}
	to all solutions to a general Pell's equation.

	\begin{lem}
		\label{gen_pell_eqn_EDT0L_lem}
		Let \(S\) be the set of all solutions to the general Pell's
		equation \(X^2 - DY^2 = N\), and \(\alpha, \ \beta, \ \gamma, \ \delta, \
		\epsilon, \ \zeta \in \mathbb{Z}\). Then
		\begin{enumerate}
			\item The language \(L = \{a^{\alpha x + \beta y + \gamma} \# b^{\delta x
			+ \epsilon y + \zeta} \mid (x, \ y) \in S\}\) is EDT0L;
			\item A \(\#\)-separated EDT0L system \(\mathcal H\) for \(L\) is
			constructible in \(\ns(fg)\), where \(f\) is logarithmic in
			\(\max(|\alpha|, \ |\beta|, \ |\gamma|, \ |\delta|, \ |\epsilon|, \
			|\zeta|, \ |N|, \ D)\), and \(g\) is linear in \(D\);
			\item The system \(\mathcal H\) is \(h_1 h_2\)-bounded, where \(h_1\) is
			linear in \(\max(|\alpha|, \ |\beta|, \ |\gamma|, \ |\delta|, \
			|\epsilon|, \ |\zeta|, \ |N|, \ D)\), and \(h_2\) is exponential in \(D\).
		\end{enumerate}
	\end{lem}

	\begin{proof}
		First note that the following are equivalent:
		\begin{enumerate}
			\item \((x, \ y) \in S\);
			\item \((x, \ -y) \in S\);
			\item \((-x, \ y) \in S\);
			\item \((-x, \ -y) \in S\).
		\end{enumerate}
		Since we can use Lemma \ref{hash_sep_union_lem} to take finite
		unions of EDT0L languages, and preserve space complexity of EDT0L systems,
		it therefore suffices to show that
		\[
			M = \{a^{\alpha x + \beta y + \gamma} \# b^{\delta x + \epsilon y + \zeta}
			\mid (x, \ y) \in \text{ is a non-negative integer solution to } X^2 -
			DY^2 = N\}
		\]
		is accepted by an EDT0L system that satisfies (2) and (3). Using
		\cref{Pell_non_primitive_solns_lem}, all non-negative integer solutions to
		\(X^2 - DY^2 = N\) are of the form \((xk, \ yk)\) where \((x, \ y)\) is a
		primitive solution to \(X^2 - DY^2 = \frac{N}{k^2}\), for some \(k\) such
		that \(k^2 | N\). Moreover, if \((x, \ y)\) is a primitive solution to \(X^2
		- DY^2 = \frac{N}{k^2}\), then \((xk, \ yk)\) is a solution to \(X^2 - DY^2
		= N\). We will therefore show two claims:
		\begin{enumerate}
			\item The language \(M_k = \{a^{\alpha k x + \beta k y + \gamma} \#
			b^{\delta kx + \epsilon ky + \zeta} \mid (x, \ y) \text{ is a primitive
			solution to } X^2 - DY^2 = \frac{N}{k^2}\}\) is EDT0L for all \(k \in
			\mathbb{Z}_{> 1}\) such that \(k^2 | N\);
			\item The union of the languages \(M_k\) is EDT0L, and accepted by a
			\(\#\)-separated EDT0L system \(\mathcal H\) that satisfies (2) and (3);
			\item The extended alphabet and endomorphisms in \(\mathcal H\) satisfy
			the conditions in (3).
		\end{enumerate}
		Since \(M\) equals this union, the result follows.

		First note that if \(k \in \mathbb{Z}_{\geq 2}\) is such that \(k^2 | N\),
		then \(k < N\). Thus \(\log(\alpha k) = \log(\alpha) + \log(k) \leq
		\log(\alpha) + \log(N)\). If we use \(\beta\), \(\delta\) or \(\epsilon\) in
		place of \(\alpha\), this inequality will still hold. Thus the first claim
		follows from \cref{gen_pell_eqn_primitive_EDT0L_lem}.

		For the second claim, we can apply \cref{hash_sep_union_lem} repeatedly,
		once for each \(k \in \mathbb{Z}_{\geq 2}\) such that \(k^2 | N\). We need
		to do this for all such \(k\). This could be done by cycling through all \(k
		\in \{2, \ \ldots, \ N - 1\}\), checking if \(k^2 | N\), and then applying
		the lemma in those cases. We would need to store the `current' \(k\) to do
		this, which would use at most \(\log(N)\) bits. \qedhere

	\end{proof}

	%

	Before attempting to tackle the general two-variable quadratic equations, we
	mention the result we can obtain so far for a general Pell's equation. The
	space complexity in this case is log-linear, which is better than the
	log-quartic space complexity we have for arbitrary two-variable quadratic
	equations.

	\begin{proposition}
		The solution language to the general Pell's equation \(X^2 - DY^2 = N\)
		is EDT0L, accepted by an EDT0L system that is constructible in
		non-deterministic log-linear space, with \(\max(D, \ |N|)\) as the input
		size.
	\end{proposition}

	\begin{proof}
		This follows by first taking \(\alpha = \beta = \delta = \epsilon = 1\)
		and \(\gamma = \zeta = 0\) in Lemma \ref{gen_pell_eqn_EDT0L_lem}, and then
		applying a free monoid homomorphism that maps \(b\) to \(a\), using Lemma
		\ref{EDT0L_closure_properties_lem}.
	\end{proof}

	In order to understand the solutions to a generic two-variable quadratic
	equation, we must first know the solutions to the equation \(X^2 + DY^2 = N\),
	where \(N, \ D \in \mathbb{Z}\). Whilst we have considered the `hardest' case
	of \(D < 0\), \(-D\) non-square and \(N \neq 0\), it remains to consider the
	remaining cases. We start with the case when \(D \geq 0\).

	\begin{lem}
		\label{negative_D_Pell_EDT0L_lem}
		Let \(S\) be the set of all solutions to the equation \(X^2 + DY^2 = N\),
		with \(N \in \mathbb{Z}\), \(D \in \mathbb{Z}_{\geq 0}\), and \(\alpha, \
		\beta, \ \gamma, \ \delta, \ \epsilon, \ \zeta \in \mathbb{Z}\). Then
		\begin{enumerate}
			\item The language \(L = \{a^{\alpha x + \beta y + \gamma} \# b^{\delta x
			+ \epsilon y + \zeta} \mid (x, \ y) \in S\}\) is EDT0L;
			\item A \(\#\)-separated EDT0L system \(\mathcal H\) for \(L\) is
			constructible in \(\ns(f)\), where \(f\) is logarithmic in
			\(\max(|\alpha|, \ |\beta|, \ |\gamma|, \ |\delta|, \ |\epsilon|, \
			|\zeta|, \ |N|, \ D)\);
			\item The system \(\mathcal H\) is \(h\)-bounded, where \(h\) is linear in
			\(\max(|\alpha|, \ |\beta|, \ |\gamma|, \ |\delta|, \ |\epsilon|, \
			|\zeta|, \ |N|, \ D)\).
		\end{enumerate}
	\end{lem}

	\begin{proof}
		If \(N < 0\), there is nothing to prove, as the equation has no solutions.
		So suppose \(N \geq 0\). Then all solutions \((x, \ y)\) to this equation
		satisfy \(|x| + |y| \leq N\). Let \((x, \ y)\) be such a solution. Then
		\(\{a^{\alpha x + \beta y + \gamma} \# b^{\delta x + \epsilon y + \zeta}\}\)
		is accepted by the EDT0L system \((\{a, \ a^{-1}, \ b, \ b^{-1}, \ \#\}, \
		\{a, \ a^{-1}, \ b, \ b^{-1}, \ \#\} \cup \{\perp_1, \ \perp_2\}, \ \perp_1
		\# \perp_2, \ \varphi)\), where \(\varphi\) is defined by
		\[
			c \varphi = \left\{
				\begin{array}{cl}
					a^{\alpha x + \beta y + \gamma} & c = \perp_1 \\
					b^{\delta x + \epsilon y + \zeta} & c = \perp_2 \\
					c & \text{otherwise}.
				\end{array}
			\right.
		\]
		We have that this EDT0L system satisfies the conditions stated in (2) and
		(3). Thus we can use Lemma \ref{hash_sep_union_lem} to obtain the result.
	\end{proof}

	We now consider the solutions to the equation \(X^2 - DY^2 = N\) when \(D\)
	is square.

	\begin{lem}
		\label{sqaure_D_Pell_EDT0L_lem}
		Let \(S\) be the set of all solutions to the equation \(X^2 - DY^2 = N\),
		with \(N \in \mathbb{Z}\), \(D \in \mathbb{Z}_{\geq 0}\), such that
		\(D\) is a perfect square. Let \(\alpha, \
		\beta, \ \gamma, \ \delta, \ \epsilon, \ \zeta \in \mathbb{Z}\). Then
		\begin{enumerate}
			\item The language \(L = \{a^{\alpha x + \beta y + \gamma} \# b^{\delta x
			+ \epsilon y + \zeta} \mid (x, \ y) \in S\}\) is EDT0L;
			\item A \(\#\)-separated EDT0L system \(\mathcal H\) for \(L\) is
			constructible in \(\ns(f)\), where \(f\) is logarithmic in
			\(\max(|\alpha|, \ |\beta|, \ |\gamma|, \ |\delta|, \ |\epsilon|, \
			|\zeta|, \ |N|, \ D)\);
			\item The system \(\mathcal H\) is \(h\)-bounded, where \(h\) is linear in
			\(\max(|\alpha|, \ |\beta|, \ |\gamma|, \ |\delta|, \ |\epsilon|, \
			|\zeta|, \ |N|, \ D)\).
		\end{enumerate}
	\end{lem}

	\begin{proof}
		As \(D\) is square, we have that \(D = E^2\) for some \(E \in
		\mathbb{Z}_{\geq 0}\). Define a new variable \(V = EY\). Substituting this
		into \(X^2 - DY^2 = N\) gives \(X^2 - V^2 = N\), which is again equivalent
		to \((X - V)(X + V) = N\). If \((x, \ v)\) is a solution, then \(x + v\) and
		\(x - v\) must both divide \(N\), and thus \(|x + v| \leq |N|\) and \(|x -
		v| \leq |N|\). It follows that there are finitely many solutions \((x, \
		v)\) to \(X^2 - V^2 = N\), all of which satisfy \(|x| \leq |N|\) and \(|v|
		\leq N\). Thus there are finitely many solutions \((x, \ y)\) to \(X^2 -
		D Y^2 = N\), all of which satisfy \(|x| \leq |N|\) and \(|y| \leq |V| \leq
		|N|\).

		We can use the same argument we used in Lemma
		\ref{negative_D_Pell_EDT0L_lem}, to show that each of the singleton
		languages \(\{a^{\alpha x + \beta y + \gamma} \# b^{\delta x + \epsilon y +
		\zeta}\}\) are accepted by EDT0L systems satisfying the conditions in (2)
		and (3). We can again use Lemma \ref{hash_sep_union_lem} to union these to
		form \(L\).
	\end{proof}

	We finally need to consider the case when \(N = 0\).

	\begin{lem}
		\label{0_N_Pell_EDT0L_lem}
		Let \(S\) be the set of all solutions to the equation \(X^2 - DY^2 = 0\),
		with \(D \in \mathbb{Z}_{\geq 0}\), and \(\alpha, \
		\beta, \ \gamma, \ \delta, \ \epsilon, \ \zeta \in \mathbb{Z}\). Then
		\begin{enumerate}
			\item The language \(L = \{a^{\alpha x + \beta y + \gamma} \# b^{\delta x
			+ \epsilon y + \zeta} \mid (x, \ y) \in S\}\) is EDT0L;
			\item A \(\#\)-separated EDT0L system \(\mathcal H\) for \(L\) is
			constructible in \(\ns(f)\), where \(f\) is logarithmic in
			\(\max(|\alpha|, \ |\beta|, \ |\gamma|, \ |\delta|, \ |\epsilon|, \
			|\zeta|, \ D)\);
			\item The system \(\mathcal H\) is \(h\)-bounded, where \(h\) is linear in
			\(\max(|\alpha|, \ |\beta|, \ |\gamma|, \ |\delta|, \ |\epsilon|, \
			|\zeta|, \ |N|, \ D)\).
		\end{enumerate}
	\end{lem}

	\begin{proof}
		First note that \((x, \ y)\) is a solution if and only if \(x = \sqrt{D}
		y\). It follows that if \(D\) is non-square, then this admits no solutions,
		and there is nothing to prove. If \(D\) is square, then the result follows
		from Lemma \ref{sqaure_D_Pell_EDT0L_lem}.
	\end{proof}

	Combining the different cases of the equation \(X^2 + DY^2 = N\) allows us to
	use Lagrange's method to show the following. The proof follows the arguments
	given in \cite{new_look_old_equation}, Section 1.

	\begin{theorem}
		\label{quad_eqn_2_var_integers_EDT0L_thm}
		Let
		\begin{equation}
			\label{two_var_quad_eqn}
			\alpha X^2 + \beta XY + \gamma Y^2 + \delta X + \epsilon Y + \zeta = 0
		\end{equation}
		be a two-variable quadratic
		equation in the ring of integers, with a set \(S\) of solutions. Then
		\begin{enumerate}
			\item The language \(L = \{a^x \# b^y \mid (x, \ y) \in S\}\) is EDT0L;
			\item Taking the input size to be \(\max(|\alpha|, \ |\beta|, \ |\gamma|,
			\ |\delta|, \ |\epsilon|, \ |\zeta|)\), an EDT0L system for \(L\) is
			constructible in \(\ns(n \mapsto n^4 \log n)\).
		\end{enumerate}
	\end{theorem}

	\begin{proof}
		Let \(D = \beta^2 - 4 \alpha \gamma\), \(E = \beta \delta - 2 \alpha
		\epsilon\) and \(F = \delta^2 - 4 \alpha \zeta\), and define
		new variables \(U = DY + E\) and \(V = 2 \alpha X + \beta Y + \delta\). Then
		\begin{enumerate}
			\item \(V^2 = 4 \alpha^2 X^2 + \beta^2 Y^2 + \delta^2 + 4 \alpha \beta XY
			+ 4 \alpha \delta X + 2 \beta \delta Y\);
			\item \(DY^2 = \beta^2 Y^2 - 4 \alpha \gamma Y^2\);
			\item \(2EY = 2 \beta \delta Y - 4 \alpha \epsilon Y\);
			\item \(F = \delta^2 - 4 \alpha \zeta\).
		\end{enumerate}
		Thus
		\[
			V^2 - DY^2 - 2EY - F = 4 \alpha^2 X^2 + 4 \alpha \beta XY + 4 \alpha
			\delta X + 4 \alpha \gamma Y^2 + 4 \alpha \epsilon Y + 4 \alpha \zeta.
		\]
		It follows that \eqref{two_var_quad_eqn} can be rewritten as
		\begin{equation*}
			V^2 = DY^2 + 2EY + F.
		\end{equation*}
		This is equivalent to
		\[
			DV^2 = (DY + E)^2 + DF - E^2.
		\]
		By substituting \(U\) for \(DY + E\), and setting \(N = E^2 - DF\), we can
		conclude that \eqref{two_var_quad_eqn} can be written as
		\begin{align}
			\label{gen_pell_in_proof_reduction_eqn}
			U^2 - DV^2 = N.
		\end{align}

		Note that
		\[
			Y = \frac{U - E}{D}, \qquad X = \frac{V - \beta Y - \delta}{2\alpha}
			= \frac{VD - \beta U + \beta E - \delta D}{2 \alpha D}.
		\]
		Let \(T\) be the set of solutions to
		\eqref{gen_pell_in_proof_reduction_eqn}. By Lemma
		\ref{gen_pell_eqn_EDT0L_lem}, Lemma \ref{negative_D_Pell_EDT0L_lem}, Lemma
		\ref{sqaure_D_Pell_EDT0L_lem} or Lemma \ref{0_N_Pell_EDT0L_lem} (dependent
		on whether \(D\) is positive and non-square, positive and square, or
		non-positive, and whether or not \(N = 0\)) we have that
		\[
			 M = \{a^{DV - \beta U + \beta E - \delta D} \# b^{U - E} \mid (u, \
			 v) \in T\}
		\]
		is accepted by a \(\#\)-separated EDT0L system \(\mathcal H\), which is
		constructible in \(\ns(fg)\), where \(f\) is logarithmic in \(\max(|D|, \
		|\beta|, \ |\beta E - \delta D|, \ |E|)\), and \(g\) is linear in \(|D|\).
		Let \(C\) be the extended alphabet of \(\mathcal H\), and let \(B\) be the
		finite set of endomorphisms of \(C^\ast\) over which the rational control of
		\(\mathcal H\) is regular. Using Lemma \ref{gen_pell_eqn_EDT0L_lem},
		\ref{gen_pell_eqn_EDT0L_lem}, Lemma \ref{negative_D_Pell_EDT0L_lem}, Lemma
		\ref{sqaure_D_Pell_EDT0L_lem} or Lemma \ref{0_N_Pell_EDT0L_lem}, we also
		have that \(|C|\) and \(\max\{|c \phi| \mid c \in C, \ \phi \in B\}\) are
		bounded by \(h_1 h_2\), where \(h_1\) is linear in \(\max(|D|, \ |\beta|, \
		|\beta E - \delta D|, \ |E|)\), and \(h_2\) is exponential in \(|D|\).

		We have that \(D = \beta^2 - 4 \alpha \gamma\) and \(N = E^2 - DF = (\beta
		\delta - 2 \alpha \epsilon)^2 - (\beta^2 - 4 \alpha \gamma)(\delta^2 - 4
		\gamma \zeta)\). It follows that \(f\) is logarithmic in \(\max(|\alpha|, \
		|\beta|, \ |\gamma|, \ |\delta|, \ |\epsilon|)\), and \(g\) is quadratic in
		\(\max(|\alpha|, \ |\beta|, \ |\gamma|)\).

		In addition, \(h_1\) is quartic in \(\max(|\alpha|, \ |\beta|, \ |\gamma|, \
		|\delta|, \ |\epsilon|)\), and \(h_2\) is \(\mathcal O(2^{n^2})\) in
		\(\max(|\alpha|, \ |\beta|, \ |\gamma|)\).

		Note that \(DY = U - E\) and \(2 \alpha D X = DV - \beta U + \beta E -
		\delta D\). Thus we have that
		\[
			M = \{a^{2 \alpha D x} \# b^{Dy} \mid (x, \ y) \in S\}.
		\]
		By \cref{EDT0L_division_lem}, it follows that \(L\) is EDT0L, and
		accepted by an EDT0L system that is constructible in \(\ns(n \mapsto
		n^4 \log n)\).
	\end{proof}

	Using Lemma \ref{EDT0L_closure_properties_lem} to apply the free monoid
	homomorphism that maps \(b\) to \(a\) to a language described in Theorem
	\ref{quad_eqn_2_var_integers_EDT0L_thm} gives the following:

	\begin{cor}
		\label{quad_eqn_2_var_integers_EDT0L_cor}
		The solution language to a two-variable quadratic equation in integers is
		EDT0L, accepted by an EDT0L system that is constructible in \(\ns(n \mapsto
		n^4 \log n)\), with the input size taken to be the maximal absolute value of
		a coefficient.
	\end{cor}

\section{From Heisenberg equations to integer equations}
  \label{Heisenberg_eqns_sec}
  This section aims to prove that the solution language to an equation in one
  variable in the Heisenberg group is EDT0L. We do this by showing that	a single
  equation \(\mathcal{E}\) in the Heisenberg group is `equivalent' to a system
  \(S_\mathcal{E}\) of quadratic equations in the ring of integers. The idea of
  the proof is to replace each variable in \(\mathcal{E}\) with a word
  representing a potential solution, and then convert the resulting word into
  Mal'cev normal form. The equations in \(S_\mathcal{E}\) occur by equating the
  exponent of the generators to \(0\).

	We start with an example of an equation in the Heisenberg group.

  \begin{ex}
    \label{Heisenberg_group_eqn_ex}
    We will transform the equation \(XYX = 1\) in the Heisenberg group into a
    system over the integers. Using the Mal'cev normal form we can write \(X =
    a^{X_1} b^{X_2} c^{X_3}\) and \(Y = a^{Y_1} b^{Y_2} c^{Y_3}\) for
    variables \(X_1, \ X_2, \ X_3, \ Y_1, \ Y_2, \ Y_3\) over the integers.
    Replacing \(X\) and \(Y\) in \(XYX = 1\) in these expressions gives
    \begin{align}
      \label{Heisenberg_ex_init_eqn}
      a^{X_1} b^{X_2} c^{X_3} a^{Y_1} b^{Y_2} c^{Y_3} a^{X_1}
      b^{X_2} c^{X_3} = 1.
    \end{align}
    After manipulating this into Mal'cev normal form, we obtain
    \begin{align}
      \label{Heisenberg_ex_malcev_eqn}
      a^{2 X_1 + Y_1} b^{2X_2 + Y_2} c^{2 X_3 + Y_3 + X_1 Y_2 + X_1 X_2 +
      Y_1 X_2} = 1.
    \end{align}
    As this normal form word is trivial if and only if the exponents of
    \(a\), \(b\) and \(c\) are all equal to \(0\), we obtain the following
    system
    over \(\mathbb{Z}\):
    \begin{align}
      \label{Heisenberg_ex_Z_sys_eqn}
      & 2 X_1 + Y_1 = 0 \\
      \nonumber
      & 2 X_2 + Y_2 = 0 \\
      \nonumber
      & 2 X_3 + Y_3 + X_1 Y_2 + X_1 X_2 + Y_1 X_2 = 0.
    \end{align}
    Note that the variables corresponding to the exponent of \(c\) in \(X\)
    and \(Y\), namely \(X_3\) and \(Y_3\), only appear in linear terms in the
    above system.

    In this specific example it is not hard to enumerate the solutions in a
    somewhat reasonable manner. We can start by replacing occurrences of
    \(Y_1\) and \(Y_2\) in the third equation of
    \eqref{Heisenberg_ex_Z_sys_eqn} with \(-2X_1\) and \(-2X_2\),
    respectively, to give that \eqref{Heisenberg_ex_Z_sys_eqn} is equivalent
    to
    \begin{align*}
      & Y_1 = -2X_1 \\
      \nonumber
      & Y_2 = -2X_2 \\
      \nonumber
      & 2 X_3 + Y_3 - 2X_1 X_2 + X_1 X_2 - 2 X_1 X_2 = 0.
    \end{align*}
    This simplifies to
    \begin{align*}
      & Y_1 = -2X_1 \\
      \nonumber
      & Y_2 = -2X_2 \\
      \nonumber
      & 2 X_3 + Y_3 =  3X_1 X_2.
    \end{align*}
    We can now enumerate all values of \((X_1, \ X_2, \ X_3)\) (across
    \(\mathbb{Z})\), and each such choice will fix the values of \(Y_1\),
    \(Y_2\) and \(Y_3\), for which there will always exist a solution. Using
    this method, we have that the solution set to
    \eqref{Heisenberg_ex_Z_sys_eqn} is equal to
    \[
      \{(x_1, \ x_2, \ x_3, \ -2x_1, \ -2x_2, \ 3x_1 x_2 - 2x_3) \mid x_1, \
      x_2, \ x_3 \in \mathbb{Z}\}.
    \]
    Translating this back into the language of the Heisenberg group gives that
    the solution set to \(XYX = 1\) is
    \[
      \{(a^{x_1} b^{x_2} c^{x_3}, \ a^{-2x_1} b^{-2x_2} c^{3x_1 x_2 - 2x_3})
      \mid x_1, \ x_2, \ x_3 \in \mathbb{Z}\}.
    \]
  \end{ex}


	The following definition allows us to transform an equation in a single
	variable in the Heisenberg group into a system of equations in the ring of
	integers. This is done by representing the variables as expressions in Mal'cev
	normal form, plugging these expressions back into the equation, and then
	converting the resulting word into Mal'cev normal form. After doing this, the
	exponents of the generators can the be equated to \(0\), which yields a
	system of equations in the ring of integers.

	\begin{dfn}
		If \(w = 1\) is an equation in a class 2 nilpotent group, consider
		the system of equations over the integers defined by taking the variable
		\(X\), and viewing it in Mal'cev normal form by introducing new variables:
		\(X = a^{X_1} b^{Y_1} c^{Z_1}\), where the \(X_1\), \(X_2\) and \(X_3\)
		take values in \(\mathbb{Z}\). The resulting system of equations over
		\(\mathbb{Z}\) obtained by setting the expressions in the exponents equal
		to zero is called the \textit{\(\mathbb{Z}\)-system} of \(w = 1\).
	\end{dfn}

	\begin{ex}
		The \(\mathbb{Z}\)-system of the equation \eqref{Heisenberg_ex_init_eqn}
		from Example \ref{Heisenberg_group_eqn_ex} is
		\begin{align*}
			& 2 X_1 + Y_1 = 0 \\
			& 2 X_2 + Y_2 = 0 \\
			& 2 X_3 + Y_3 + X_1 Y_2 + X_1 X_2 + Y_1 X_2 = 0.
		\end{align*}
	\end{ex}

	We now explicitly calculate the \(\mathbb{Z}\)-system of an arbitrary
	equation in one variable in the Heisenberg group.

	\begin{lem}
		\label{Z_system_lem}
		Let
		\begin{align}
			\label{Heisenberg_eqn}
			X^{\epsilon_1} a^{i_1} b^{j_1} c^{k_1} \cdots X^{\epsilon_n} a^{i_n}
			b^{j_n} c^{k_n} = 1
		\end{align}
		be a single equation in one variable in the Heisenberg group, where
		\(\epsilon_1, \ \ldots, \ \epsilon_n \in \{-1, \ 1\}\), and \(i_1, \
		\ldots i_n, \ j_1, \ \ldots, \ j_n, \ k_1, \ \ldots, \ k_n \in
		\mathbb{Z}\). Define
		\[
			\delta_r = \left\{
			\begin{array}{cl}
				0 & \epsilon_r = 1 \\
				1 & \epsilon_r = -1.
			\end{array}
			\right.
		\]
		Writing \(X = a^{X_1} b^{X_2} c^{X_3}\) with \(X_1\),
		\(X_2\) and \(X_3\) over \(\mathbb{Z}\) gives that the
		\(\mathbb{Z}\)-system of \eqref{Heisenberg_eqn} is
		\begin{align*}
			& \sum_{r = 1}^n (\epsilon_r X_1 + i_r) = 0 \\
			&	\sum_{r = 1}^n (\epsilon_r X_2 + j_r) = 0 \\
			& \sum_{r = 1}^n (\epsilon_r X_3 + k_r + \delta_r X_1
			X_2) + \sum_{r = 1}^n \sum_{s = 1}^r ( \epsilon_r \epsilon_s X_1 X_2 +
			\epsilon_r X_1 j_s) + \sum_{r = 1}^n \sum_{s = 1}^r (i_r \epsilon_s X_2
			+ i_r j_s) = 0.
		\end{align*}
	\end{lem}

	\begin{proof}
		We proceed as in Example \ref{Heisenberg_group_eqn_ex}. Replacing each
		occurrence of \(X\) in \eqref{Heisenberg_eqn} with \(a^{X_1} b^{X_2}
		c^{X_3}\) gives
		\begin{align}
			(a^{X_1} b^{X_2} c^{X_3})^{\epsilon_1} a^{i_1} b^{j_1} c^{k_1}
			\cdots (a^{X_1} b^{X_2} c^{X_3})^{\epsilon_n} a^{i_n} b^{j_n} c^{k_n}
			= 1.
		\end{align}
		Since \(c\) is central, we can push all occurrences of \(c\) and \(c^{-1}\)
		to the right, and then freely reduce, thus showing that \ref{Heisenberg_eqn}
		is equivalent to
		\begin{align}
			\label{Z_system_Heisenberg_eqn_2}
			(a^{X_1} b^{X_2})^{\epsilon_1} a^{i_1} b^{j_1}
			\cdots (a^{X_1} b^{X_2})^{\epsilon_n} a^{i_n} b^{j_n} c^{\sum_{r = 1}^n
			(\epsilon_r X_3 + k_r)}
			= 1.
		\end{align}
		Note that for all \(x_1, \ x_2 \in \mathbb{Z}\), \((a^{x_1} b^{x_2})^{-1}
		= b^{-x_2} a^{-x_1} = a^{-x_1} b^{-x_2} c^{x_1 x_2}\). Using this,
		together with the fact that \(c\) is central, gives that
		\eqref{Z_system_Heisenberg_eqn_2} is equivalent to
		\begin{align}
			\label{Z_system_Heisenberg_eqn_3}
			a^{\epsilon_1 X_1} b^{\epsilon_1 X_2} a^{i_1} b^{j_1} \cdots
			a^{\epsilon_n X_1} b^{\epsilon_n X_2} a^{i_n} b^{j_n} c^{\sum_{r = 1}^n
			(\epsilon_r X_3 + k_r + \delta_r X_1 X_2)}
			= 1.
		\end{align}
		We now push all \(a\)s in \eqref{Z_system_Heisenberg_eqn_3} to the left.
		The \(a\)s at the beginning do not need to move. The \(a\)s with exponent
		\(i_1\) will need to move past \(b^{\epsilon_1 X_2}\), thus increasing
		the exponent of \(c\) by \(i_1 \epsilon_1 X_2\). The \(a\)s with
		exponent \(\epsilon_2 X_1\) will need to move past \(b^{j_1}\) and
		\(b^{\epsilon_1 X_2}\), thus increasing the exponent of \(c\) by
		\(j_1 \epsilon_2 X_1 + \epsilon_1 \epsilon_2 X_1 X_2\). This continues
		up to the \(a\)s with exponent \(i_n\), which will need to move past
		all \(b\)s, thus increasing the exponent of \(c\) by \(i_n (\sum_{r = 1}^n
		\epsilon_r X_2) + i_n \sum_{r = 1}^{n - 1} j_r\). Overall, we have that
		\eqref{Z_system_Heisenberg_eqn_3} is equivalent to
		\begin{align}
			\label{Z_system_Heisenberg_eqn_4}
			& a^{\displaystyle \sum_{r = 1}^n (\epsilon_r X_1 + i_r)} \\
			\nonumber
			&	b^{\displaystyle \sum_{r = 1}^n (\epsilon_r X_2 + j_r)} \\
			\nonumber
	 		& c^{\displaystyle \sum_{r = 1}^n (\epsilon_r X_3 + k_r + \delta_r X_1
	 		X_2) + \sum_{r = 1}^n \sum_{s = 1}^r ( \epsilon_r \epsilon_s X_1 X_2 +
	 		\epsilon_r X_1 j_s) + \sum_{r = 1}^n \sum_{s = 1}^r (i_r \epsilon_s X_2
	 		+ i_r j_s)}
			= 1.
		\end{align}
		Equating each of the exponents to \(0\) (as we are now in Mal'cev normal
		form) gives that the \(\mathbb{Z}\)-system of \eqref{Heisenberg_eqn} is
		\begin{align*}
			& \sum_{r = 1}^n (\epsilon_r X_1 + i_r) = 0 \\
			&	\sum_{r = 1}^n (\epsilon_r X_2 + j_r) = 0 \\
			& \sum_{r = 1}^n (\epsilon_r X_3 + k_r + \delta_r X_1
			X_2) + \sum_{r = 1}^n \sum_{s = 1}^r ( \epsilon_r \epsilon_s X_1 X_2 +
			\epsilon_r X_1 j_s) + \sum_{r = 1}^n \sum_{s = 1}^r (i_r \epsilon_s X_2
			+ i_r j_s) = 0. \qedhere
		\end{align*}
	\end{proof}

	We have now collected the results we need to prove the main theorem of this
	section.

	\begin{theorem}
		\label{Heisenberg_eqns_EDT0L_thm}
		Let \(L\) be the solution language to a single equation with one
		variable in the Heisenberg group, with respect to the Mal'cev generating
		set and normal form. Then
		\begin{enumerate}
			\item The language \(L\) is EDT0L;
			\item An EDT0L system for \(L\) is constructible in \(\ns(n \mapsto
			n^8(\log n)^2)\).
		\end{enumerate}
	\end{theorem}

	\begin{proof}
		Let
		\begin{equation}
		\label{Heisenberg_eqn_v2}
			X^{\epsilon_1} a^{i_1} b^{j_1} c^{k_1} \cdots X^{\epsilon_n} a^{i_n}
			b^{j_n} c^{k_n} = 1
		\end{equation}
		be an equation in the Heisenberg group in a single variable. By Lemma
		\ref{Z_system_lem}, we have that the \(\mathbb{Z}\)-system of
		\eqref{Heisenberg_eqn_v2} is
		\begin{align}
			\label{Z_system_Heisenberg_final}
			& \sum_{r = 1}^n (\epsilon_r X_1 + i_r) = 0 \\
			\nonumber
			&	\sum_{r = 1}^n (\epsilon_r X_2 + j_r) = 0 \\
			\nonumber
			& \sum_{r = 1}^n (\epsilon_r X_3 + k_r + \delta_r X_1
			X_2) + \sum_{r = 1}^n \sum_{s = 1}^r ( \epsilon_r \epsilon_s X_1 X_2 +
			\epsilon_r X_1 j_s) + \sum_{r = 1}^n \sum_{s = 1}^r (i_r \epsilon_s X_2
			+ i_r j_s) = 0.
		\end{align}
		We consider two cases: when \(\sum_{r = 1}^n \epsilon_r = 0\) and when
		\(\sum_{r = 1}^n \epsilon_r \neq 0\).

		Case 1: \(\sum_{r = 1}^n \epsilon_r = 0\). \\
		Applying our case assumption to \eqref{Z_system_Heisenberg_final} gives
		that \eqref{Z_system_Heisenberg_final} is equivalent to
		\begin{align}
			\label{Z_system_EDT0L_1}
			& \sum_{r = 1}^n i_r = 0 \\
			\nonumber
			&	\sum_{r = 1}^n j_r = 0 \\
			\nonumber
			& \sum_{r = 1}^n (k_r + \delta_r X_1
			X_2) + \sum_{r = 1}^n \sum_{s = 1}^r ( \epsilon_r \epsilon_s X_1 X_2 +
			\epsilon_r X_1 j_s) + \sum_{r = 1}^n \sum_{s = 1}^r (i_r \epsilon_s X_2
			+ i_r j_s) = 0.
		\end{align}
		The first two of the above identities only involve constants. If one of
		these is not satisfied, then \eqref{Heisenberg_eqn_v2} has no solutions.
		In such a case, \(L\) is empty, and there is nothing to prove. So we
		suppose that these are satisfied. It follows that they are redundant, and
		the above system is equivalent to the third equation in it (with the
		addition that \(X_3\) can be anything, regardless of \(X_1\) and \(X_2\)).
		Note that this is a quadratic equation in integers, with variables \(X_1\)
		and \(X_2\). So by Theorem \ref{quad_eqn_2_var_integers_EDT0L_thm}
		\[
			K = \{a^{x_1} \# b^{x_2} \mid (x_1, \ x_2) \text{ is part of a solution
			\eqref{Z_system_EDT0L_1} for } (X_1, \ X_2)\}
		\]
		is EDT0L, and accepted by an EDT0L system that is constructible in \(\ns(n
		\mapsto n^4 \log n)\) in terms of the coefficients of the equation. These
		are
		\[
			\sum_{r = 1}^n k_r + \sum_{r = 1}^n \sum_{s = 1}^r i_r j_s, \ \sum_{r =
			1}^n \delta_r + \sum_{r = 1}^n \sum_{s = 1}^r \epsilon_r \epsilon_s, \
			\sum_{r = 1}^n \sum_{s = 1}^r \epsilon_r j_s, \ \sum_{r = 1}^n \sum_{s =
			1}^r i_r \epsilon_s.
		\]
		Note that \(|\epsilon_r| = 1\) and \(|\delta_r| \leq 1\) for all \(r\). In
		addition, as exponents of constants in \eqref{Heisenberg_eqn_v2}, each sum
		\(\sum_{r = 1}^n i_r\), \(\sum_{r = 1}^n j_r\) and \(\sum_{r = 1}^n k_r\)
		is linear in our input. It follows that the above expression is quadratic
		in our input, and so an EDT0L system for \(K\) is constructible in \(\ns(n
		\mapsto n^8(\log n)^2)\). Applying the monoid homomorphism that maps
		\(\#\) to \(\varepsilon\), followed by concatenating the above language
		with the EDT0L language \(\{c\}^\ast\), which is constructible in constant
		space, allows us to apply Lemma \ref{EDT0L_closure_properties_lem} to show
		\[
			\{a^{x_1} b^{x_2} c^{x_3} \mid (x_1, \ x_2, \ x_3) \text{ is  a solution
			\eqref{Z_system_EDT0L_1}}\}
		\]
		is EDT0L, accepted by an EDT0L system that is constructible in \(\ns(n
		\mapsto n^8(\log n)^2)\). Since this language is \(L\), the result follows.

		Case 2: \(\sum_{r = 1}^n \epsilon_r \neq 0\). \\
		Let \(\alpha = \sum_{r = 1}^n \epsilon_r\), \(\beta = \sum_{r = 1}^n i_r\),
		\(\gamma = \sum_{r = 1}^n j_r\) and \(\zeta = \sum_{r = 1}^n k_r\). Then
		we can rewrite \eqref{Z_system_Heisenberg_final} as
		\begin{align}
			\label{Z_system_EDT0L_2}
			& \alpha X_1 + \beta = 0 \\
			\nonumber
			&	\alpha X_2 + \gamma = 0 \\
			\nonumber
			& \alpha X_3 + \zeta + \sum_{r = 1}^n \delta_r X_1
			X_2 + \sum_{r = 1}^n \sum_{s = 1}^r ( \epsilon_r \epsilon_s X_1 X_2 +
			\epsilon_r X_1 j_s) + \sum_{r = 1}^n \sum_{s = 1}^r (i_r \epsilon_s X_2
			+ i_r j_s) = 0.
		\end{align}
		If either of the first two equations have no solution, then neither
		does \eqref{Heisenberg_eqn_v2}, and so \(L\) is empty, and there is
		nothing to prove. We will therefore suppose that both of these equations
		admit a solution. Since these are both single linear equations with one
		variable, they can both admit a single solution. Let \(x_1\) be the
		solution for \(X_1\), and \(x_2\) be the solution for \(X_2\). Plugging
		these into the third equation gives
		\begin{align}
			\label{Z_system_EDT0L_3}
			\alpha X_3 + \zeta + \sum_{r = 1}^n \delta_r x_1
			x_2 + \sum_{r = 1}^n \sum_{s = 1}^r (\epsilon_r \epsilon_s x_1 x_2 +
			\epsilon_r x_1 j_s) + \sum_{r = 1}^n \sum_{s = 1}^r (i_r \epsilon_s x_2
			+ i_r j_s) = 0.
		\end{align}
		Note that this is a linear equation in integers with single variable
		\(X_3\). Hence by \cite{VAEP}, Corollary 3.13 and Proposition 3.16, the
		language
		\[
			M = \{c^{x_3} \mid x_3 \text{ is a solution to }
			\eqref{Z_system_EDT0L_3}\}
		\]
		is EDT0L, and accepted by an EDT0L system that is constructible in
		non-deterministic quadratic space in terms of an input of length
		\[
			|\alpha| + |\zeta| + \sum_{r = 1}^n |\delta_r x_1 x_2| +
			\sum_{r = 1}^n \sum_{s = 1}^r (|\epsilon_r \epsilon_s x_1 x_2| +
			|\epsilon_r x_1 j_s|) + \sum_{r = 1}^n \sum_{s = 1}^r (|i_r \epsilon_s
			x_2| + |i_r j_s|).
		\]
		As the sums of the lengths of constants in our original equation,
		\(|\alpha|\), \(|\beta|\), \(|\gamma|\) and \(|\zeta|\) are all linear in
		our input. As the number of constants in our equation, \(n\) is also linear
		in our input. We have that \(|x_1| = \left|\frac{\beta}{\alpha} \right| \leq
		|\beta|\) and \(|x_2| = \left| \frac{\gamma}{\alpha} \right| \leq |\alpha|\)
		are both linear in our input. Since \(|\epsilon_r| = 1\) and \(|\delta_r|
		\leq 1\) for all \(r\), and the above expression is quartic in our input, it
		follows that \(M\) is constructible in \(\ns(n \mapsto n^4)\). Applying
		Lemma \ref{EDT0L_closure_properties_lem} to concatenate \(M\) with the
		singleton language \(\{a^{x_1} b^{x_2}\}\), which is constructible in linear
		space, gives that
		\[
			\{a^{x_1} b^{x_2} c^{x_3} \mid (x_1, \ x_2, \ x_3) \text{ is a solution
			to } \eqref{Z_system_EDT0L_1}\}
		\]
		is EDT0L, and accepted by an EDT0L system that is constructible in
		\(\ns(n \mapsto n^4)\). Since this language is \(L\), the
		result follows.
	\end{proof}

\section*{Acknowledgments}
  I would like to thank Laura Ciobanu for incredibly helpful mathematical
  discussions and writing advice. I would like to thank Ross Paterson for
  answering various number theoretic questions, and for running simulations to
  test some ideas. I would like to thank Luke Elliott for some helpful comments
  on an earlier version. I would like to thank the reviewer for some helpful
  comments. I would like to thank the London Mathematical Society, the
  Heilbronn Institute for Mathematical Research	and the University of St
  Andrews for their support during the writing of this paper.

\nocite{OEIS}
\bibliography{references}
\bibliographystyle{abbrv}
\end{document}